\documentclass[11pt]{amsart}
\usepackage{amssymb}
\usepackage{color}
\usepackage{amsmath}
\usepackage{graphicx}

\providecommand{\U}[1]{\protect\rule{.1in}{.1in}}

\newtheorem{theorem}{Theorem}[section]
\newtheorem{lemma}[theorem]{Lemma}
\newtheorem{corollary}[theorem]{Corollary}
\newtheorem{proposition}[theorem]{Proposition}

\theoremstyle{definition}
\newtheorem{example}[theorem]{Example}
\newtheorem{definition}[theorem]{Definition}
\theoremstyle{remark}
\newtheorem{remark}[theorem]{Remark}

\begin{document}
\title{Kashiwara and Zelevinsky involutions in affine type $A$}
\date{}
\author{Nicolas Jacon and C\'edric Lecouvey}
\maketitle

\begin{abstract}
We describe how the Kashiwara involution $\ast $ on crystals of affine type $
A$ is encoded by the combinatorics of aperiodic multisegments. This permits
to prove in an elementary way that $\ast $ coincides with the Zelevinsky
involution $\tau $ on the set of simple modules for the affine Hecke
algebras. We then give efficient procedures for computing $\ast $ and $\tau $
. Remarkably, these procedures do not use the underlying crystal structure.
They also permit to match explicitly the Ginzburg and Ariki parametrizations
of the simple modules associated to affine and cyclotomic Hecke algebras,
respectively .
\end{abstract}

\section{Introduction}

The Kashiwara involution $\ast $ in affine type $A$ is a fundamental
anti-\linebreak isomorphism of the quantum group $\mathcal{U}_{v}$
associated to the affine root system $A_{e-1}^{(1)}.$ It induces a subtle
involution on $B_{e}(\infty )$, the Kashiwara crystal corresponding to the
negative part $\mathcal{U}_{v}^{-}$ of $\mathcal{U}_{v}$. The Zelevinsky
involution yields an involution $\tau $ of the affine Hecke algebra of type $
A$. When $q$ is specialized to an $e$-th root of $1,$ $\tau $ also induces
an involution on $B_{e}(\infty ).$ In this paper, we show by using the
combinatorics of aperiodic multisegments that the Kashiwara and Zelevinsky
involutions coincide on $B(\infty )$. We also provide efficient procedures
for computing these involutions. In addition, our results permit to match
explicitly the Ginzburg and Ariki parametrizations of the simple modules
associated to affine and cyclotomic Hecke algebras respectively. All our
computations can be made independent of the crystal structure on $
B_{e}(\infty )$. Moreover, they do not require the determination of $i$
-induction or $i$-restriction operations on simple modules. Let us now
describe the context and the results of the paper more precisely.

The Zelevinsky involution first appeared in \cite{Ze} in connection with the
representation theory of the linear group $GL(n,\mathbb{F}_{\mathfrak{p}})$
over the $\mathfrak{p}$-adic field {\ $\mathbb{F}_{\mathfrak{p}}$}. Works by
Moeglin and Waldspurger \cite{MW} then permit to link it with a natural
involution $\tau $ of the affine type $A$ Hecke algebra $\mathcal{H}
_{n}^{a}(q)$ over the field $\mathbb{F}$ with generic parameter $q$. {When $
e\geq 2$ is an integer and }$q${\ is specialized at $\xi $, a primitive $e$
-root of $1$,} it was conjectured by Vigneras \cite{Vi} that this involution
should be related to the modular representation theory of $GL(n,\mathbb{F}_{ 
\mathfrak{p}})$. In the sequel we will refer to $\tau $ has the Zelevinsky
involution of $\mathcal{H}_{n}^{a}(\xi )$ (see Section \ref{Sec_Ze} for a
complete definition).

 The involution $\tau$ induces an involution on the set of simple $
\mathcal{H}_{n}^{a}(\xi)$-modules. There exist essentially two different
parametrizations of these modules in the literature. In the geometric
construction of Chriss and Ginzburg \cite{CG} and under the assumption $
\mathbb{F=C}$, the simple $\mathcal{H}_{n}^{a}(\xi)$-modules are labelled by
aperiodic multisegments. These simple modules can also be regarded as simple
modules associated to Ariki-Koike algebras $\mathcal{H}_{n}^{\mathbf{v}}(\xi
)$. The Specht module theory developed by Dipper, James and Mathas then
provides a labelling of the simple $\mathcal{H}_{n}^{a}(\xi)$-modules by
Kleshchev multipartitions.\ Both constructions permit to endow the set of
simple $\mathcal{H}_{n}^{a}(\xi)$-modules with the structure of a crystal
isomorphic to $B_{e}(\infty)$. The Kashiwara crystal operators then yield
the modular branching rules for the Ariki-Koike algebras and affine Hecke
algebras of type $A$ (\cite{Ari4}, \cite{AJL}).

  In \cite{Groj}, Grojnowski uses $i$-induction and $i$-restrictions
operators to define an abstract crystal structure on the set of simple $
\mathcal{H}_{n}^{a}(\xi)$-modules. He then proves that this crystal is in
fact isomorphic to $B_{e}(\infty)$. This approach is valid over an arbitrary
field $\mathbb{F}$ and does not require the Specht module theory of Dipper
James and Mathas. This notably permits to extend the methods of \cite{Groj}
to the representation theory of the cyclotomic Hecke-Clifford superalgebras 
\cite{BK}.\ Nevertheless, this approach does not match up the abstract
crystal obtained with the labellings of the simple modules by aperiodic
multisegments or Kleshchev multipartitions. Since the $i$-induction
operation on simple modules is difficult to obtain in general, it is also
not really suited to explicit computations.

The identification of $\mathcal{U}_{v}^{-}$ with the composition subalgebra
of the Hall algebra associated to the cyclic quiver of type $A_{e}^{(1)}$
yields two different structures of crystal on the set of aperiodic
multisegments. They both come from two different parametrizations of the
canonical basis of $\mathcal{U}_{v}^{-}$ which correspond under the
anti-isomorphism $\rho $ on $\mathcal{U}_{v}^{-}$ switching the generators $
f_{i}$ and $f_{-i}$. In particular $\rho $ provides an involution on the
crystal $B_{e}(\infty )$ which can be easily computed. The use of the
composition algebra also permits to describe explicitly the structure of
Kashiwara crystal on the set of aperiodic multisegments. This was obtained
in \cite{LTV} by Leclerc, Thibon and Vasserot. In addition, these authors
prove that the involution $\tau $ on $B_{e}(\infty )$ satisfies the identity $
\tau =\sharp \circ \rho $ where $\sharp $ is the two fold symmetry on $
B_{e}(\infty )$ which switches the sign of each arrow.

In this paper, we first establish that the two crystal structures on
aperiodic multisegments obtained by identifying $\mathcal{U}_{v}^{-}$ with
the composition algebra correspond up to the conjugation by the Kashiwara
involution $\ast $.\ This implies that 
\begin{equation*}
\ast =\tau \text{ on }B_{e}(\infty ).
\end{equation*}
Observe that an equivalent identity can also be established by using results
of \cite{Groj} but, as mentioned above, it then requires subtle
considerations on representation theory of $\mathcal{H}_{n}^{a}(\xi )$ and
does not permit to compute $\ast =\tau $ efficiently. In contrast our proof
uses only elementary properties of crystal graphs and yields efficient
procedures for computing the involution $\ast =\tau $. {This notably
permits us to generalize an algorithm of Moeglin and Waldspurger which gives
the Zelevinsky involution when $e=\infty $.}

 As a consequence, {extending works of Vazirani \cite{V}, we
completely} solve the following natural problem. Given a simple $\mathcal{H}
_{n}^{a}(\xi )$-module $L_{\psi }$ (with $\psi $ an aperiodic multisegment $
\psi $), {\ we find all the Ariki-Koike algebras $\mathcal{H}_{n}^{\mathbf{v}
}(q)$ and the simple $\mathcal{H}_{n}^{\mathbf{v}}(q)$-modules $D^{{\ 
\boldsymbol{\lambda}}}$ (with $\boldsymbol{\lambda}$ a Kleshchev
multipartition) such that $D^{{\boldsymbol{\lambda}}}\simeq L_{\psi }$ as $
\mathcal{H}_{n}^{a}(\xi )$-modules.} The procedure yielding the Kashiwara
involution also permits to compute the commutor of $A_{e}^{(1)}$-crystals
introduced by Kamnitzer and Tingley in \cite{KT}.

The paper is organized as follows. In Section 2, we review the
identification of $\mathcal{U}_{v}^{-}$ with the composition algebra and the
two structures of crystal it gives on the set of aperiodic multisegments.\
We also recall basic facts on the Kashiwara involution. Section 3 is devoted
to the definition of the Zelevinsky involution on the set of simple $
\mathcal{H}_{n}^{a}(\xi )$-modules and to the results of \cite{LTV}. In
Section 4, we prove the identity $\ast =\tau .$ The problem of determining
the algebras $\mathcal{H}_{n}^{\mathbf{v}}(\xi )$ and the simple $\mathcal{H}
_{n}^{\mathbf{v}}(\xi )$-modules isomorphic to a given simple $\mathcal{H}
_{n}^{a}(\xi )$-module is studied in Section 5. In the last two sections, we
give a simple combinatorial procedure for computing the involutions $\tau
,\rho $ and $\sharp $ on $B_{e}(\infty ).$ We prove in fact that all these
computations can essentially be obtained from the Mullineux involution on $e$
-regular partitions and the crystal isomorphisms described in \cite{JL}. We
also investigate several consequences of our results.

\section{Quantum groups and crystals in affine type $A$}

\subsection{The quantum group $\mathcal{U}_{v}$}

\label{subsec_Hall}

Let $v$ be an indeterminate and $e\geq 2$ an integer. Write $\mathcal{U}
_{v}( \widehat{\mathfrak{sl}_{e}})$ for the quantum group of type $
A_{e-1}^{(1)}$. This is an associative $\mathbb{Q}(v)$-algebra with
generators $e_{i},f_{i},t_{i},t_{i}^{-1}$, $i\in \mathbb{Z}/e\mathbb{Z}$ and 
$\partial $ (see \cite[\S 2.1]{U} for the complete description of the
relations satisfied by these generators). Write $\{\Lambda _{0},...,\Lambda
_{e-1},\delta \}$ and $\{\alpha _{0},\ldots ,\alpha _{e-1}\}$ respectively
for the set of fundamental weights and the set of simple roots associated to 
${\mathcal{U}_{v}(\widehat{\mathfrak{sl}_{e}})}$. Let $P$ be the weight
lattice of ${\mathcal{U}_{v}(\widehat{\mathfrak{sl}_{e}})}$. We denote by ${
\ \mathcal{U}}_{v}={\mathcal{U}_{v}^{\prime }(\widehat{\mathfrak{sl}_{e}})}$
the subalgebra generated by $e_{i},f_{i},t_{i},t_{i}^{-1},$ $i\in \mathbb{Z}
/e\mathbb{Z}$. Then $\overline{P}=P/\mathbb{Z\delta }$ is the set of
classical weights of ${\mathcal{U}}_{v}$. For any $i\in \mathbb{Z}/e\mathbb{
Z }$, we also denote by $\Lambda _{i}$ and $\alpha _{i}$ the restriction of $
\Lambda _{i}$ and $\alpha _{i}\in P$ to $\overline{P}^{\wedge }$.\ Let ${\ 
\mathcal{U}_{v}^{-}}$ be the subalgebra of ${\mathcal{U}_{v}}$ generated by
the $f_{i}$'s with $i\in \mathbb{Z}/e\mathbb{Z}$.

\subsection{Two crystal structures on the set of aperiodic multisegments}

\label{subsec_twoactions}

\begin{definition}
Let $l\in {\mathbb{Z}}_{>0}$ and $i\in {\mathbb{Z}}/e{\mathbb{Z}}$. The 
\emph{segment of length $l$ and head $i$} is the sequence of consecutive
residues $[i,i+1,...,i+l-1]$. We denote it by $[i;l)$. Similarly, The \emph{
\ segment of length $l$ and tail $i$} is the sequence of consecutive
residues $[i-l+1,...,i-1,i]$. We denote it by $(l;i]$.
\end{definition}

\begin{definition}
A collection of segments is called a \emph{multisegment}. If the collection
is the empty set, we call it the empty multisegment and it is denoted by $
\boldsymbol{\emptyset}$.
\end{definition}

It is convenient to write a multisegment $\psi $ on the form 
\begin{equation*}
\psi =\sum_{i\in \mathbb{Z}/e\mathbb{Z},l\in \mathbb{N}_{>0}}m_{[i;l)}[i;l).
\end{equation*}

\begin{definition}
A multisegment $\psi $ is \emph{aperiodic} if, for every $l\in {\mathbb{Z}}
_{>0}$, there exists some $i\in {\mathbb{Z}}/e{\mathbb{Z}}$ such that $(l;i]$
does not appear in $\psi $. Equivalently, a multisegment $\psi $ is
aperiodic if, for each $l\in {\mathbb{Z}}_{>0}$, there exists some $i\in {\ 
\mathbb{Z}}/e{\mathbb{Z}}$ such that $[i;l)$ does not appear in $\psi $. We
denote by $\Psi _{e}$ the set of aperiodic multisegments.
\end{definition}

Let $B_{e}(\infty )$ be the (abstract) crystal basis of $\mathcal{U}_{v}^{-}$
. By results of Ringel and Lusztig, the algebra ${\mathcal{U}_{v}^{-}}$ is
isomorphic to the composition algebra of the Hall algebra associated to the
cyclic quiver $\Gamma _{e}$ of length $e.$ This yields in particular a
natural parametrization of the vertices of $B_{e}(\infty )$ by $\Psi _{e}.$
We can thus regard the vertices of $B_{e}(\infty )$ as aperiodic
multisegments.\ The corresponding crystal structure was described by
Leclerc, Thibon and Vasserot in \cite[Theorem 4.1]{LTV}. In fact we shall
need in the sequel two different structures of crystal on $\Psi _{e}.$ They
are linked by the involution $\rho $ which negates all the segments of a
given multisegment, that is such that 
\begin{equation}
\psi ^{\rho }=\sum_{i\in \mathbb{Z}/e\mathbb{Z},l\in \mathbb{N}
_{>0}}m_{[i;l)}(l;-i]  \label{ro_psi}
\end{equation}
for any multisegment $\psi =\sum_{i\in \mathbb{Z}/e\mathbb{Z},l\in \mathbb{N}
_{>0}}m_{[i;l)}[i;l)$. The involution $\rho $ as a natural algebraic
interpretation since it also yields a linear automorphism of the Hall
algebra associated to $\Gamma _{e}.$ Since we do not use Hall algebras in
this paper, we only recall below the two crystal structures relevant for our
purpose.

Let $\psi $ be a multisegment and let $\psi _{\geq l}$ be the multisegment
obtained from $\psi $ by deleting the multisegments of length less than $l$,
for $l\in {\mathbb{Z}}_{>0}$. Denote by $m_{[i;l)}$ the multiplicity of $
[i;l)$ in $\psi $. For any $i\in \mathbb{Z}/e\mathbb{Z}$, set 
\begin{equation*}
\widehat{S}_{l,i}=\sum_{k\geq l}(m_{[i+1;k)}-m_{[i;k)}).
\end{equation*}
Let $\widehat{l}_{0}$ be the minimal value of $l$ that attains $\min_{l>0} 
\widehat{S}_{l,i}$.

\begin{theorem}
\label{LTV-1} Let $\psi $ be a multisegment, $i\in {\mathbb{Z}}/e{\mathbb{Z}}
$ and let $\widehat{l}_{0}$ be as above. Then we have 
\begin{equation*}
\widehat{f}_{i}\psi =\psi _{\widehat{l}_{0},i},
\end{equation*}
where the multisegment $\psi _{\widehat{l}_{0},i}$ is defined as follows 
\begin{equation*}
\psi _{\widehat{l}_{0},i}=\left\{ 
\begin{array}{ll}
\psi +[i;1) & \text{ if }\widehat{l}_{0}=1, \\ 
\psi +[i;\widehat{l}_{0})-[i+1;\widehat{l}_{0}-1) & \text{ if }\widehat{l}
_{0}>1.
\end{array}
\right.
\end{equation*}
\end{theorem}

The crystal structure on $\Psi _{e}$ obtained from the action of the
operators $\widehat{f}_{i},$ $i\in \mathbb{Z}/e\mathbb{Z}$ does not coincide
with that initially described by Leclerc, Thibon and Vasserot.\ The LTV
crystal structure stated in \cite{LTV} is obtained by using the crystal
operators 
\begin{equation}
\widetilde{f}_{i}=\rho \circ \widehat{f}_{-i}\circ \rho, \qquad i\in \mathbb{Z}/e 
\mathbb{Z}  \label{rel_tilde_hat}
\end{equation}
rather than the operators $\widehat{f}_{i}$. More precisely, set $
S_{l,i}=\sum_{k\geq l}(m_{(k;i-1]}-m_{(k;i]})$.\ Let $l_{0}$ be the minimal $
l$ that attains $\min_{l>0}S_{l,i}$. Then, the crystal structure
corresponding to the $\widetilde{f}_{i}$'s is given as follows.

\begin{theorem}
\label{LTV-2} Let $\psi $ be a multisegment and let $i\in {\mathbb{Z}}/e{\ 
\mathbb{Z}}$ and $l_{0}$ be as above. Then we have 
\begin{equation*}
\widetilde{f}_{i}\psi =\psi _{l_{0},i},
\end{equation*}
where the multisegment $\psi _{{l}_{0},i}$ is defined as follows 
\begin{equation*}
\psi _{l_{0},i}=\left\{ 
\begin{array}{ll}
\psi +(1;i] & \text{ if }{l}_{0}=1, \\ 
\psi +({l}_{0};i]-({l}_{0}-1;i-1] & \text{ if }{l}_{0}>1.
\end{array}
\right.
\end{equation*}
\end{theorem}

Let $\psi$ be a multisegment. Then to compute $\widetilde{e}_{i}\psi$, we
proceed as follows. If $\min_{l>0}S_{l,i}=0$, then $\widetilde{e}_{i}\psi=0$
. Otherwise, let $l_{0}$ be the maximal $l$ that attains $\min_{l>0}S_{l,i}$
. Then, $\widetilde{e}_{i}\psi$ is obtained from $\psi$ by replacing $
(l_{0};i]$ with $(l_{0}-1;i-1]$.

In the sequel, we identify $B_{e}(\infty )$ with the crystal structure
obtained on $\Psi _{e}$ by considering the operators $\widetilde{f}_{i},$ $
i\in \mathbb{Z}/e\mathbb{Z}$ (see also Remark \ref{rem}). Then $\rho $
induces an involution on $B_{e}(\infty )$ and the crystal operators $
\widetilde{f}_{i}$ and $\widehat{f}_{i}$ are related by (\ref{rel_tilde_hat}
).We denote by $\mathrm{wt}(\psi )$ the weight of the aperiodic multisegment 
$\psi $ considered as a vertex of the crystal $B_{e}(\infty ).$ Set 
\begin{equation}
\mathrm{wt}(\psi )=\sum_{i\in \mathbb{Z}/e\mathbb{Z}}\mathrm{wt}_{i}(\psi )\
\Lambda _{i}.  \label{weight}
\end{equation}
For any $i\in \mathbb{Z}/e\mathbb{Z}$, define $\varepsilon _{i}(\psi )=\max
\{k\in \mathbb{N}\mid \widetilde{e}_{i}^{k}(u)\neq 0\}$ and $\varphi
_{i}(\psi )=\mathrm{wt}_{i}(\psi )+\varepsilon _{i}(\psi )$.

\subsection{The Kashiwara involution}

The Kashiwara involution $\ast$ is the $\mathcal{U}_{v} (\widehat {\mathfrak{
\ sl}_{e}})$-antiautomorphism such that $q^{\ast}=q$ and defined on the
generators as follows: 
\begin{equation}
{e}_{i}^{\ast}={e}_{i},\quad{f}_{i}^{\ast}={f}_{i},\quad
t_{i}^{\ast}=t^{-1}_{i}.\quad
\end{equation}

Since $\ast$ stabilizes ${\mathcal{U}_{v}^{-}}$, it induces an involution
(also denoted $\ast$) on $B_{e}(\infty)$ the crystal graph of ${\mathcal{U}
_{v}^{-}}$.\ By setting for any vertex $b\in B_{e}(\infty)$ and any $i\in 
\mathbb{Z}/e\mathbb{Z}$ 
\begin{equation}
\widetilde{e}_{i}^{\ast}(b)=\widetilde{e}_{i}(b^{\ast})^{\ast},\quad 
\widetilde{f}_{i}^{\ast}(b)=\widetilde{f}_{i}(b^{\ast})^{\ast},\quad
\varepsilon_{i}^{\ast}(b)=\varepsilon_{i}(b^{\ast})\quad\text{and }\varphi
_{i}^{\ast}(b)=b^{\ast}  \label{*crystal}
\end{equation}
we obtain another crystal structure on $B_{e}(\infty)$ (see \cite{Kas}).

Let $i\in\mathbb{Z}/e\mathbb{Z}$ and write $B_{i}$ for the crystal with set
of vertices $\{b_{i}(k)\mid k\in\mathbb{Z}\}$ and such that 
\begin{align*}
\mathrm{wt(}b_{i}(k)) & =k\alpha_{i},\quad\varepsilon_{j}(b_{i}(k))=\left\{ 
\begin{array}{c}
-k\text{ if }i=j \\ 
-\infty\text{ if }i\neq j
\end{array}
\right. ,\quad\varphi_{j}(b_{i}(k))=\left\{ 
\begin{array}{c}
k\text{ if }i=j \\ 
-\infty\text{ if }i\neq j
\end{array}
\right. \\
\widetilde{e}_{j}b_{i}(k) & =\left\{ 
\begin{array}{l}
b_{i}(k+1)\text{ if }i=j \\ 
0\text{ if }i\neq j
\end{array}
\right. \quad\text{and }\quad\widetilde{f}_{j}b_{i}(k)=\left\{ 
\begin{array}{l}
b_{i}(k-1)\text{ if }i=j \\ 
0\text{ if }i\neq j
\end{array}
\right. .
\end{align*}
Set $b_{i}=b_{i}(0).$

\noindent Recall the action of $\widetilde{e}_{i}$ and $\widetilde{f}_{i}$
on the tensor product $B\otimes B^{\prime}=\{b\otimes b^{\prime}\mid b\in
B,b^{\prime}\in B^{\prime}\}$ of the crystals $B$ and $B^{\prime}$ :

\begin{align}
\widetilde{f_{i}}(u\otimes v)& =\left\{ 
\begin{tabular}{c}
$\widetilde{f}_{i}(u)\otimes v$ if $\varphi _{i}(u)>\varepsilon _{i}(v),$ \\ 
$u\otimes \widetilde{f}_{i}(v)$ if $\varphi _{i}(u)\leq \varepsilon _{i}(v),$
\end{tabular}
\ \right.  \label{TENS1} \\
& \text{and}  \notag \\
\widetilde{e_{i}}(u\otimes v)& =\left\{ 
\begin{tabular}{c}
$u\otimes \widetilde{e_{i}}(v)$ if $\varphi _{i}(u)<\varepsilon _{i}(v),$ \\ 
$\widetilde{e_{i}}(u)\otimes v$ if $\varphi _{i}(u)\geq \varepsilon _{i}(v).$
\end{tabular}
\ \right.  \label{TENS2}
\end{align}
The embedding of crystals $\theta _{i}:B_{e}(\infty )\hookrightarrow
B_{e}(\infty )\otimes B_{i}$ which sends the highest weight vertex $b_{ 
\boldsymbol{\emptyset}}$ of $B_{e}(\infty )$ on $b_{\boldsymbol{\emptyset}
}\otimes B_{i}$ permits to compute the action of the operators $\widetilde{e}
_{i}^{\ast }$ and $\widetilde{f}_{i}^{\ast }$ at least on a theoretical
point of view.

\begin{proposition}
\label{Prop_ac*}(see Proposition 8.1 in \cite{Kas}) Consider $b\in
B_{e}(\infty )$ and set $\varepsilon _{i}^{\ast }(b)=m$.\ Then we have

\begin{enumerate}
\item $\theta _{i}(b)=(\widetilde{e}_{i}^{\ast })^{m}b\otimes \widetilde{f}
_{i}^{m}b_{i},$

\item $\theta _{i}(\widetilde{f}_{i}^{\ast }b)=(\widetilde{e}_{i}^{\ast
})^{m}b\otimes \widetilde{f}_{i}^{m+1}b_{i}$ and

\item $\theta _{i}(\widetilde{e}_{i}^{\ast }b)=(\widetilde{e}_{i}^{\ast
})^{m}b\otimes \widetilde{f}_{i}^{m-1}b_{i}$ if $m>0$ and $\theta _{i}( 
\widetilde{e}_{i}^{\ast }b)=0$ if $m=0.$
\end{enumerate}
\end{proposition}

\begin{remark}
\label{rem}\ 

\begin{enumerate}
\item By \ref{subsec_twoactions}, $\Psi _{e}$ is equipped with two crystal
structures.\ One is obtained from the action of the crystal operators $
\widehat{f}_{i},i\in \mathbb{Z}/e\mathbb{Z}$ and the other one is related to
the operators $\widetilde{f}_{i},i\in \mathbb{Z}/e\mathbb{Z}$ and yields the
Kashiwara crystal graph structure $B_{e}(\infty )$ on $\Psi _{e}$. We shall
see in Section \ref{Sec_hat=star}, that the actions of the operators $
\widehat{f}_{i}$ and $\widetilde{f}_{i}^{\ast }$ with $i\in \mathbb{Z}/e\mathbb{
Z }$ coincide.

\item Observe that Proposition \ref{Prop_ac*} does not provide an efficient
procedure for computing the involution $\ast $. Indeed, in order to obtain $
\theta _{i}(b)$, we have first to determine a path from $b$ to the highest
weight vertex of $B_{e}(\infty ).$ Moreover, computing a section of the
embedding $\theta _{i}$ is difficult in general.
\end{enumerate}
\end{remark}

\subsection{Crystals of highest weight ${\mathcal{U}_{v}}$-modules}

\label{subsec_FLOW}

Let $l\in\mathbb{N}$ and consider ${\mathbf{v}}=(v_{0},\ldots,v_{l-1})\in 
\mathbb{Z}^{l}$. ${\mathbf{v}}$ is called a multicharge and $l$ is by
definition the level of ${\mathbf{v}}$. One can then associate to $\mathbf{v}
$ the abstract ${\mathcal{U}_{v}}$-irreducible module $V_{e}(\Lambda _{ 
\mathbf{v}})$ with highest weight $\Lambda_{\mathbf{v}}=\Lambda _{v_{0}( 
\mathrm{mod}e)}+\cdots+\Lambda_{v_{l-1}(\mathrm{mod}e)}$. There exist
distinct realizations of $V_{e}(\Lambda_{\mathbf{v}})$ as an irreducible
component of a Fock space $\mathfrak{F}_{e}^{{\mathbf{v}}}$ whose structure
depends on $\mathbf{v}$.\ As a $\mathbb{C(}v)$-vector space, the Fock space $
\mathfrak{F}_{e}^{{\mathbf{v}}}$ of level $l$ admits the set of all $l$
-partitions as a natural basis. Namely the underlying vector space is 
\begin{equation*}
\mathfrak{F}_{e}=\bigoplus_{n\geq0}\bigoplus_{{{\boldsymbol{\lambda}}}
\in\Pi_{l,n}}\mathbb{C}(v)\mathbf{{\boldsymbol{\lambda}}}
\end{equation*}
where $\Pi_{l,n}$ is the set of $l$-partitions with rank $n$. Consider $
\mathfrak{v=(v}_{0},\ldots,\mathfrak{v}_{l-1})\in(\mathbb{Z}/e\mathbb{Z)}
^{l} $.\ We write $\mathbf{v\in}\mathfrak{v}$ when $v_{c}\in\mathfrak{v}_{c} 
$ for any $c=0,\ldots,l-1$.\ As ${\mathcal{U}_{v}}$-modules, the Fock spaces 
$\mathfrak{F}_{e}^{\mathbf{v}},$ $\mathbf{v\in}\mathfrak{v}$ are all
isomorphic but with distinct actions for ${\mathcal{U}_{v}}$. For each of
these actions, the empty $l$-partition $\boldsymbol{\emptyset}=(\emptyset
,\ldots,\emptyset)$ is a highest weight vector of highest weight $\Lambda_{{
\mathbf{v}}}$. We denote by $V_{e}(${$\mathbf{v}$}$)$ the irreducible
component with highest weight vector $\boldsymbol{\emptyset}$ in $\mathfrak{
F }_{e}^{{\mathbf{v}}}$. The modules $V_{e}(${$\mathbf{v}$}$)$ when {$
\mathbf{v }$} runs over $\mathfrak{v}$ are all isomorphic to the abstract
module $V_{e}(\Lambda_{\mathbf{v}})$. However, the actions of the Chevalley
operators on these modules do not coincide in general.

The module $\mathfrak{F}_{e}^{{\mathbf{v}}}$ admits a crystal graph $B_{e}^{{
\mathbf{v}}}$ labelled by $l$-partitions.\ Let us now recall the crystal
structures on $B_{e}^{{\mathbf{v}}}$ and $B_{e}({\mathbf{v)}}$ the crystal
associated to $V_{e}(${$\mathbf{v}$}$)$. We will omit the description of the 
${\mathcal{U}_{v}}$-module structures on $\mathfrak{F}_{e}^{{\mathbf{v}}}$
and $V_{e}(${$\mathbf{v}$}$)$ which are not needed in our proofs (see \cite
{jim} for a complete exposition).

\noindent Let ${\boldsymbol{\lambda}}=(\lambda^0,...,\lambda^{l-1})$ be an $l$-partition (identified with
its Young diagram). Then, the nodes of ${\boldsymbol{\lambda}}$ are the
triplets $\gamma=(a,b,c)$ where $c\in\{0,\ldots,l-1\}$ and $a,b$ are
respectively the row and column indices of the node $\gamma$ in $\lambda
^{c}.$ The content of $\gamma$ is the integer $c\left( \gamma\right)
=b-a+v_{c}$ and the residue $\mathrm{res(}\gamma)$ of $\gamma$ is the
element of $\mathbb{Z}/e\mathbb{Z}$ such that 
\begin{equation}
\mathrm{res}(\gamma)\equiv c(\gamma)(\mathrm{mod}e).  \label{res}
\end{equation}
We say that $\gamma$ is an $i$-node of ${\boldsymbol{\lambda}}$ when $
\mathrm{res}(\gamma)\equiv i(\mathrm{mod}e).$ This node is removable when $
\gamma=(a,b,c)\in{\boldsymbol{\lambda}}$ and ${\boldsymbol{\lambda}}
\backslash\{\gamma\}$ is an $l$-partition. Similarly $\gamma$ is addable
when $\gamma=(a,b,c)\notin{\boldsymbol{\lambda}}$ and ${\boldsymbol{\lambda}}
\cup\{\gamma\}$ is an $l$-partition.

\noindent The structure of crystal on $B_{e}^{{\mathbf{v}}}$ (and in fact,
the structure of $\ {\mathcal{U}_{v}}$-module on $\mathfrak{F}_{e}^{\mathbf{
v }}$ itself) is conditioned by the total order $\prec_{{\mathbf{v}}}$ on
the set of addable and removable $i$-nodes of the multipartitions. Consider $
\gamma _{1}=(a_{1},b_{1},c_{1})$ and $\gamma_{2}=(a_{2},b_{2},c_{2})$ two $i 
$-nodes in ${\boldsymbol{\lambda}}$. We define the order $\prec_{{\mathbf{v}
}}$ by setting

\begin{equation*}
\gamma_{1}\prec_{{\mathbf{v}}}\gamma_{2}\Longleftrightarrow\left\{ 
\begin{array}{l}
c(\gamma_{1})<c(\gamma_{2})\text{ or} \\ 
c(\gamma_{1})=c(\gamma_{2})\text{ and }c_{1}>c_{2}.
\end{array}
\right.
\end{equation*}
Starting from any $l$-partition ${\boldsymbol{\lambda}}$, consider its set
of addable and removable $i$-nodes. Let $w_{i}$ be the word obtained first
by writing the addable and removable $i$-nodes of ${\boldsymbol{\lambda}}$
in {increasing} order with respect to $\prec_{{\mathbf{v}}}$ next by
encoding each addable $i$-node by the letter $A$ and each removable $i$-node
by the letter $R$.\ Write $\widetilde{w}_{i}=A^{r}R^{s}$ for the word
derived from $w_{i}$ by deleting as many of the factors $RA$ as possible. If 
$r>0,$ let $\gamma$ be the rightmost addable $i$-node in $\widetilde{w}_{i}$.\ When $\widetilde{w}_{i}\neq\emptyset$, the node $\gamma$ is called the
good $i$-node.

\begin{proposition}
\label{prop_crystal_muli}The crystal graph $B_{e}^{{\mathbf{v}}}$ of $
\mathfrak{F}_{e}^{{\mathbf{v}}}$ is the graph with

\begin{enumerate}
\item vertices: the $l$-partitions,

\item edges: $\displaystyle{{\boldsymbol{\lambda}}}\overset{i}{\rightarrow }{{\boldsymbol{\mu}}}$ if and only if ${\boldsymbol{\mu}}$ is obtained by
adding to ${\boldsymbol{\lambda}}$ its good $i$-node.

\item for any $i\in \mathbb{Z}/e\mathbb{Z}$, $\varepsilon _{i}({{\boldsymbol{
\lambda})=s}}$ and $\varphi _{i}({{\boldsymbol{\lambda})=r.}}$
\end{enumerate}
\end{proposition}

Since $V_{e}(${$\mathbf{v}$}$)$ is the irreducible module with highest
weight vector $\boldsymbol{\emptyset}$ in $\mathfrak{F}_{e}^{{\mathbf{v}}} 
$, its crystal graph $B_{e}(${$\mathbf{v}$}$)$ can be realized as the
connected component of highest weight vertex $\boldsymbol{\emptyset}$ in $
B_{e}^{\boldsymbol{s}}$. The vertices of $B_{e}(${$\mathbf{v}$}$)$ are
labelled by $l$-partitions called Uglov $l$-partitions associated to {$
\mathbf{v}$.}

\noindent Set 
\begin{equation}
\mathcal{V}_{l}=\{\mathbf{v}=(v_{0},\ldots,v_{l-1})\in\mathbb{Z}^{l}\mid
v_{0}\leq\cdots\leq v_{l-1}\text{ and }v_{l-1}-v_{0}<e\}.  \label{vFLOTW}
\end{equation}

\begin{definition}
\label{f} Assume that ${\mathbf{v}}\in \mathcal{V}_{l}$.\ The $l$-partition $
{\boldsymbol{\lambda}}=(\lambda ^{0},\ldots ,\lambda ^{l-1})$ is a FLOTW $l$
-partition associated to $\mathbf{v}$ if it satisfies the two following
conditions :

\begin{enumerate}
\item for all $i=1,2,\cdots $, we have : 
\begin{equation*}
\lambda _{i}^{j}\geq {\lambda _{i+v_{j+1}-v_{j}}^{j+1}}\text{ for all }
j=0,\ldots ,l-2\text{ and }\lambda _{i}^{l-1}\geq {\lambda
_{i+e+v_{0}-v_{l-1}}^{0}}.
\end{equation*}

\item for all $k>0$, among the residues appearing in ${\boldsymbol{\lambda}} 
$ at the right ends of the length $k$ rows , at least one element of $
\{0,1,\cdots ,e-1\}$ does not occur.
\end{enumerate}
\end{definition}

The following result has been obtained by Jimbo, Misra,  Miwa and  Okado
 \cite{jim} but the presentation  we adopt here comes from 
\cite{FLOTW}. 

\begin{proposition}
When ${\mathbf{v}}\in \mathcal{V}_{l}$, the set of vertices of $
B_{e}({\mathbf{v)}}$ coincides with the set of FLOTW $l$-partitions
associated to ${\mathbf{v}}$.
\end{proposition}

Let us denote by $\Phi_{e}({\mathbf{v)}}$ the set of FLOTW $l$-partitions
associated to ${\mathbf{v.\;}}$

\bigskip

Consider $\mathbf{v}\in\mathcal{V}_{l}$ and ${{\boldsymbol{\lambda}}}\in
\Phi_{e}(\mathbf{v})$. We associate to each non zero part $\lambda_{i}^{c}$
of ${{\boldsymbol{\lambda}}}$ the segment 
\begin{equation}
\lbrack(1-i+v_{c})(\mathrm{mod}e),(2-i+v_{c})(\mathrm{mod}e),\ldots
,(\lambda_{i}^{c}-i+v_{c})(\mathrm{mod}e)].  \label{f_v}
\end{equation}
The multisegment $f_{\mathbf{v}}({{\boldsymbol{\lambda}}})$ is then the
formal sum of all the segments associated to the parts $\lambda_{i}^{c}$ of $
{{\boldsymbol{\lambda}}}$. Since $f_{\mathbf{v}}({{\boldsymbol{\lambda}}}) $
is aperiodic by $(2)$ of Definition \ref{f}, the map: 
\begin{equation*}
f_{\mathbf{v}}:B_{e}(\mathbf{v})\rightarrow\Psi_{e}
\end{equation*}
is well-defined.

\begin{example}
Let $e=4$, we consider the FLOTW bipartition $(2.1,1)$ associated to $
\mathbf{v}=(0,1)$ then 
\begin{equation*}
f_{\mathbf{v}}(2.1,1)=[0,1]+[3]+[1].
\end{equation*}
Let $\mathbf{v}=(0,1,3)$ and consider the FLOTW $3$-partition $(2,1,1)$. We
have : 
\begin{equation*}
f_{\mathbf{v}}(2,1,1)=[0,1]+[1]+[3].
\end{equation*}
\end{example}

Let $T_{\Lambda}=\{t_{\Lambda}\}$ be the crystal defined by $\textrm{wt}
(t_{\Lambda_{\mathbf{v}}})=\Lambda$, $\epsilon_{i}(t_{\Lambda_{\mathbf{v}
}})=\varphi_{i}(t_{\Lambda_{\mathbf{v}}})=-\infty$ and $\widetilde{e}
_{i}t_{\Lambda_{\mathbf{v}}}=\widetilde{f}_{i}t_{\Lambda_{\mathbf{v}}}=0$.
We have a unique crystal embedding $B_{e}(\mathbf{v})\hookrightarrow
B_{e}(\infty)\otimes T_{\Lambda}.\;$The following theorem has been
established in \cite{AJL}.

\begin{theorem}
\label{Th_fv}For any $\mathbf{v}\in \mathcal{V}_{l}$, the map $f_{\mathbf{v}
} $ coincides with the unique crystal embedding $B_{e}(\mathbf{v}
)\hookrightarrow B_{e}(\infty )\otimes T_{\Lambda }$.
\end{theorem}

\noindent According to Proposition 8.2 in \cite{Kas}, we have 
\begin{equation*}
f_{\mathbf{v}}(\Phi_{e}(\mathbf{v}))=\{\psi\in\Psi_{e}\mid\varepsilon_{i}(
\psi^{\ast})\leq r_{i}\text{ for any }i\in\mathbb{Z}/e\mathbb{Z\}}
\end{equation*}
where $r_{i}$ is the number of coordinates in $\mathbf{v}$ equal to $i$ and $
\psi^{\ast}$ is the image of $\psi$ under the Kashiwara involution of the
crystal $B_{e}(\infty).$

\noindent Given any $\psi\in\Psi_{e},$ write $\mathbf{v}(\psi)$ for the
element of $\mathcal{V}_{l}$ defined by the conditions 
\begin{equation}
r_{i}=\varepsilon_{i}(\psi^{\ast})=\varepsilon_{i}^{\ast}(\psi)\text{ for
any }i\in\mathbb{Z}/e\mathbb{Z}\text{.}  \label{v(m)}
\end{equation}
Then, by the previous considerations, there exists a unique $l$-partition 
\underline{$\boldsymbol{\lambda}$}$(\psi){=}$\linebreak${(\lambda}
^{0},\ldots,\lambda^{l-1}){\in\Phi}_{e}(\mathbf{v}(\psi))$ such that $f_{ 
\mathbf{v}(\psi)}(${\underline{$\boldsymbol{\lambda}$}$(\psi$}$){)=}\psi$.

\section{The Zelevinsky involution of $\mathcal{H}_{n}^{a}(q)$}

\label{Sec_Ze}

\subsection{Three natural involutions on $\mathcal{H}_{n}^{a}(q)$}

Denote by $\mathcal{H}_{n}(q)$ the Hecke algebra of type $A$ with parameter $
q$ over the field $\mathbb{F}$.\ This is the unital associative $\mathbb{F}$
-algebra generated by $T_{1}$,..., $T_{n-1}$ and the relations:\ 
\begin{align*}
& T_{i}T_{i+1}T_{i}=T_{i+1}T_{i}T_{i+1}\ (i=1,\ldots ,n-2), \\
& T_{i}T_{j}=T_{j}T_{i}\ (|j-i|>1), \\
& (T_{i}-q)(T_{i}+1)=0\ (i=1,\ldots ,n-1).
\end{align*}
The affine Hecke algebra $\mathcal{H}_{n}^{a}(q)$ is the $\mathbb{F}$
-algebra which as an $\mathbb{F}$-module is isomorphic to 
\begin{equation*}
\mathcal{H}_{n}(q)\otimes _{R}\mathbb{F}[X_{1}^{\pm 1},\ldots ,X_{n}^{\pm
1}].
\end{equation*}
The algebra structure is obtained by requiring that $\mathcal{H}_{n}(q)$ and 
$\mathbb{F}[X_{1}^{\pm 1},\ldots ,X_{n}^{\pm 1}]$ are both subalgebras and
for any $i=1,\ldots ,n$ 
\begin{equation*}
T_{i}X_{i}T_{i}=qX_{i+1},\quad T_{i}X_{j}=X_{j}T_{i}\text{ if }i\neq j.
\end{equation*}
In the sequel, we assume that $q=\xi $ is a primitive $e$-th root of the
unity and write $\mathcal{H}_{n}^{a}(\xi )$ for the affine Hecke algebra
with parameter $\xi $. We have
three involutive automorphisms $\tau $, $\flat $ and $\sharp $ on $\mathcal{
H }_{n}^{a}(\xi )$. There are defined on the generators as follows : 
\begin{equation}
\begin{array}{ll}
T_{i}^{\tau }=-\xi T_{n-i}^{-1}, & \ X_{j}^{\tau }=X_{n+1-j}. \\ 
T_{i}^{\flat }=T_{n-i}, & \ X_{j}^{\flat }=X_{n+1-j}^{-1}. \\ 
T_{i}^{\sharp }=-\xi T_{i}^{-1}, & \ X_{j}^{\sharp }=X_{j}^{-1}.
\end{array}
\label{3invol}
\end{equation}
The involution $\sharp $ has been considered by Iwahori and Matsumoto \cite
{IM} and the involution $\tau $, which is called the Zelevinsky involution,
by Moeglin and Waldspurger \cite{MW}. One can easily check that they are
connected as follows: 
\begin{equation}
\forall x\in \mathcal{H}_{n}^{a}(\xi ),\ x^{\tau }=(x^{\flat })^{\sharp
}=(x^{\sharp })^{\flat }.  \label{relationdiese_tau_bemol}
\end{equation}

\subsection{The involutions $\protect\tau ,\sharp $ and $\flat $ on $
B_{e}(\infty )$}

We denote by $\mathrm{Mod}_{n}^{a}$ the category of finite-dimensional $
\mathcal{H}_{n}^{a}(\xi )$-modules such that for $j=1,\ldots ,n$ the
eigenvalues of $X_{j}$ are power of $\xi .$

For any multisegment $\psi =\sum_{i\in \mathbb{Z}/e\mathbb{Z},l\in \mathbb{N}
_{>0}}m_{(l;i]}(l;i],$ we write 
\begin{equation*}
\left| \psi \right| =\sum_{i\in \mathbb{Z}/e\mathbb{Z},l\in \mathbb{N}
_{>0}}lm_{(l;i]}.
\end{equation*}
The geometric realization of $\mathcal{H}_{n}^{a}(\xi )$ due to Ginzburg
permits to label the simple $\mathcal{H}_{n}^{a}(\xi )$-modules in $\mathrm{
\ Mod}_{n}^{a}$ by the aperiodic multisegments such that $\left| \psi
\right| =n.\;$We do not use Ginzburg's construction in the sequel and just
refer to \cite{CG} (see also \cite{AJL} and \cite{LTV}) for a complete
exposition or a short review. Let $L_{\psi }$ be the simple $\mathcal{H}
_{n}^{a}(\xi )$ -module corresponding to $\psi $ under this parametrization.

\noindent The three involutions $\tau ,\flat ,\sharp $ on $\mathcal{H}
_{n}^{a}(\xi )$ induce involutions on the set of simple $\mathcal{H}
_{n}^{a}(\xi )$-modules that we will denote in the same way. This yields
involutions on the set of aperiodic multisegments (also denoted by $\tau $, $
\flat $ and $\sharp $) satisfying 
\begin{equation*}
L_{\psi }^{\tau }=L_{\psi ^{\tau }},\ L_{\psi }^{\flat }=L_{\psi ^{\flat
}},\ L_{\psi }^{\sharp }=L_{\psi ^{\sharp }}
\end{equation*}
for each aperiodic multisegment $\psi $.\ Thus we have three involutions on
the vertices of $B_{e}(\infty )$. By (\ref{relationdiese_tau_bemol}), they
verify the relation 
\begin{equation}
\tau =\sharp \circ \flat =\flat \circ \sharp .
\label{relationdiese_tau_bemolB}
\end{equation}
By \cite[\S 2.4]{LTV}, for each aperiodic multisegment $\psi =\sum_{i\in 
\mathbb{Z}/e\mathbb{Z},l\in \mathbb{N}_{>0}}m_{(l;i]}[i;l)$, we have $\psi
^{\flat }=\sum_{i\in \mathbb{Z}/e\mathbb{Z},l\in \mathbb{N}
_{>0}}m_{(l;i]}(l;-i].$ By comparing with the action (\ref{ro_psi}) of $\rho 
$ on $B_{e}(\infty )$, this immediately gives the following lemma :

\begin{lemma}
\label{lemma_ro=bemol}The involutions $\rho $ and $\flat $ coincide on $
B_{e}(\infty ).$ For any $\psi =\sum_{i\in \mathbb{Z}/e\mathbb{Z},l\in 
\mathbb{N}_{>0}}m_{(l;i]}[i;l),$ we have 
\begin{equation*}
\psi ^{\flat }=\psi ^{\rho }=\sum_{i\in \mathbb{Z}/e\mathbb{Z},l\in \mathbb{
N }_{>0}}m_{(l;i]}(l;-i].
\end{equation*}
\end{lemma}

Since the action of $\rho =\flat $ on $B_{e}(\infty )$ is immediate, it is
equivalent to describe $\tau $ or $\sharp $ on $B_{e}(\infty )$. The
following proposition makes explicit the involution $\sharp $ on $
B_{e}(\infty ).$

\begin{theorem}
(\cite{LTV}) Let $\psi $ be an aperiodic multisegment. Then $\psi ^{\sharp }$
is the aperiodic multisegment obtained from $\psi $ by the $2$-fold symmetry 
$i\leftrightarrow -i$ in the graph $B_{e}(\infty )$.
\end{theorem}

\section{The equality $\ast =\protect\tau $ on $B_{e}(\infty )$}

\label{Sec_hat=star}

The aim of this section is to prove that $\ast $ and $\tau $ coincide on $
B_{e}(\infty )$. To do this, it suffices to establish the following theorem.

\begin{theorem}
\label{Th_f*=fhat}For any multisegment $\psi \in \Psi _{e}$ and any $j\in 
\mathbb{Z}/e\mathbb{Z},$ we have $\widetilde{f}_{j}^{\ast }(\psi )=\widehat{
f }_{j}(\psi ).$
\end{theorem}

Indeed, assuming this theorem holds, we easily derive the equality $\ast
=\tau .$

\begin{corollary}
The involutions $\ast $ and $\tau $ coincide on $B_{e}(\infty )$.
\end{corollary}

\begin{proof}
Let $\psi \in \Psi _{e}$. Then, there exists $i_{1},...,i_{n}$ in $\mathbb{Z}
/e\mathbb{Z}$ such that 
\begin{equation*}
\psi =\widetilde{f}_{i_{1}}...\widetilde{f}_{i_{n}}.\emptyset
\end{equation*}
Hence, we obtain 
\begin{equation*}
\begin{array}{ll}
\psi ^{\ast } & =\widetilde{f}_{i_{1}}^{\ast }...\widetilde{f}_{i_{n}}^{\ast
}.\emptyset \\ 
& =\widehat{f}_{i_{1}}...\widehat{f}_{i_{n}}.\emptyset
\end{array}
\end{equation*}
Using (\ref{rel_tilde_hat}), this gives 
\begin{equation*}
\begin{array}{ll}
\psi ^{\ast } & =\rho (\widetilde{f}_{-i_{1}}...\widetilde{f}
_{-i_{n}}.\emptyset ) \\ 
& =(\psi ^{\sharp })^{\rho }=(\psi ^{\rho })^{\sharp }=\psi ^{\tau }
\end{array}
\end{equation*}
where the two last equalities follow from Lemma \ref{lemma_ro=bemol} and ( 
\ref{relationdiese_tau_bemolB}).
\end{proof}

\bigskip

To prove Theorem \ref{Th_f*=fhat}, we are going to proceed by induction on $
\left| \psi \right| $. By using Proposition \ref{Prop_ac*}, we shall see
that it suffices in fact to establish the equivalence 
\begin{equation*}
\widetilde{f}_{i}\widehat{f}_{j}(\psi )=\widehat{f}_{j}\widetilde{f}_{i}(\psi
)\Longleftrightarrow \widetilde{f}_{i}\widetilde{f}_{j}^{\ast }(\psi )=\widetilde{f}
_{j}^{\ast }\widetilde{f}_{i}(\psi )
\end{equation*}
for any $i,j\in \mathbb{Z}/e\mathbb{Z}$. Now, by definition, the operator $
\widetilde{f}_{i}$ (resp. $\widehat{f}_{j}$) adds an entry $i$ on the right end
(resp. an entry $j$ at the left end) of one of the segments of $\psi $. This
will imply that

\begin{itemize}
\item $\widetilde{f}_{i}\widehat{f}_{j}(\psi )=\widehat{f}_{j}\widetilde{f}_{i}(\psi
)$ except possibly when $i=j$ and $\widetilde{f}_{i}(\psi )=\psi +[i]$ or $
\widehat{f}_{i}(\psi )=\psi +[i].$
\end{itemize}

On the other hand, it is easy to derive from Proposition \ref{Prop_ac*} that

\begin{itemize}
\item $\widetilde{f}_{i}\widetilde{f}_{j}^{\ast }(\psi )=\widetilde{f}_{j}^{\ast } 
\widetilde{f}_{i}(\psi )$ except possibly when $i=j$ and $\widetilde{f}
_{i}(\psi )=\psi +[i]$ or $\widetilde{f}_{i}^{\ast }(\psi )=\psi +[i].$
\end{itemize}

The case where the operators do not commute being easily tractable, this
will imply theorem \ref{Th_f*=fhat}.

\subsection{More on the crystal operators $\widetilde{f}_{i}$ and $\widehat {
f}_{i}$}

\label{subsec_cancel}

We begin with refinements of the actions of the operators $\widetilde{f}_{i}$
and $\widehat{f}_{i}.$ In \cite{AJL}, we have obtained an alternative
description of the action of the crystal operators on $\Psi _{e}.\;$Consider 
$\psi \in \Psi _{e}$ and $i\in \mathbb{Z}/e\mathbb{Z}$.\ We encode the
segments in $\psi $ with \emph{tail} $i$ (resp. $i-1$) by the symbol $R$
(resp.\ by the symbol $A$). For any nonnegative integer $l$, write $
w_{i,l}=R^{m_{(l;i]}}A^{m_{(l;i-1]}}$ where $m_{(l;i]}$ and $m_{(l;i-1]}$
are respectively the number of segments $(l;i]$ and $(l;i-1]$ in $\psi .$
Set $w_{i}=\prod_{l\geq 1}w_{i,l}.$ Write $\widetilde{w}_{i}=A^{a_{i}(\psi
)}R^{r_{i}(\psi )}$ for the word derived from $w_{i}$ by deleting as many of
the factors $RA$ as possible. If $a_{i}(\psi )>0,$ we denote by $
l_{0,i}(\psi )>0$ the length of the rightmost segment $A$ in $\widetilde{w}
_{i}$. If $a_{i}(\psi )=0,$ set $l_{0,i}(\psi )=0.$ When there is no risk of
confusion, we simply write $l_{0}$ instead of $l_{0,i}(\psi )$.

\begin{lemma}
\label{lem_util}(\cite{AJL}) With the above notation we have

\begin{enumerate}
\item $\varepsilon _{i}(\psi )=r_{i}(\psi )$

\item 
\begin{equation*}
\widetilde{f}_{i}\psi =\left\{ 
\begin{array}{l}
\psi +(l_{0};i]-(l_{0}-1;i-1]\text{ if }a_{i}(\psi )>0, \\ 
\psi +(1;i]\text{ if }a_{i}(\psi )=0.
\end{array}
\right.
\end{equation*}
\end{enumerate}
\end{lemma}

We can compute similarly the action of the crystal operators $\widehat{f}
_{i} $ (with $i\in\mathbb{Z}/e\mathbb{Z}$) on $\psi.$ We encode the segments
in $\psi$ with \emph{head} $i$ (resp. $i+1$) by the symbol $\widehat{R}$
(resp.\ by the symbol $\widehat{A}$). For any nonnegative integer $l$, write 
$\widehat{w}_{i,l}=\widehat{R}^{m_{[i;l)}}\widehat{A}^{m_{[i+1;l)}}$ where $
m_{[i;l)}$ and $m_{[i+1;l)}$ are respectively the number of segments $[i;l)$
and $[i+1;l)$ in $\psi.$ Set $\widehat{w}_{i}=\prod_{l\geq1}\widehat{w}
_{i,l}.$ Write $\overline{w}_{i}=\widehat{A}^{\widehat{a}_{i}(\psi)}\widehat{
R}^{\widehat{r}_{i}(\psi)}$ for the word derived from $\widehat {w}_{i}$ by
deleting as many of the factors $\widehat{R}\widehat{A}$ as possible. If $
\widehat{a}_{i}(\psi)>0,$ let $\widehat{l}_{0,i}(\psi)>0$ be length of the
rightmost segment $\widehat{A}$ in $\overline{w}_{i}$. If $\widehat{a}
_{i}(\psi)=0,$ set $\widehat{l}_{0,i}(\psi)=0.$ When there is no risk of
confusion, we also simply write $\widehat{l}_{0}$ instead of $\widehat{l}
_{0,i}(\psi)$.

\begin{lemma}
With the above notation, we have

\begin{enumerate}
\item $\widehat{\varepsilon }_{i}(\psi )=\widehat{r}_{i}(\psi )$ where $
\widehat{\varepsilon }_{i}(\psi )=\mathrm{max}\{p\mid \widehat{e}
_{i}^{p}(\psi )\neq 0\}.$

\item 
\begin{equation}
\widehat{f}_{i}\psi =\left\{ 
\begin{array}{l}
\psi +[i;\widehat{l}_{0})-[i+1;\widehat{l}_{0}-1)\text{ if }\widehat{a}
_{i}(\psi )>0, \\ 
\psi +[i;1)\text{ if }\widehat{a}_{i}(\psi )=0.
\end{array}
\right.  \label{act_hat_RA}
\end{equation}
\end{enumerate}
\end{lemma}

\begin{remark}
By Theorem 9.13 of \cite{Groj}, for any $i\in \mathbb{Z}/e\mathbb{Z}$, the
integer $\varepsilon _{i}(\psi )=r_{i}(\psi )$ (resp. $\widehat{\varepsilon }
_{i}(\psi )=\widehat{r}_{i}(\psi )$) gives the maximal size of a Jordan
block with eigenvalue $\xi ^{i}$ corresponding to the action of the
generator $X_{n}$ (resp. $X_{1}$) on the simple $\mathcal{H}_{n}^{a}(\xi )$
-module $L_{\psi }$.
\end{remark}

\subsection{Equality of the crystal operators $\widetilde{f}_{i}^{\ast}$ and 
$\widehat{f}_{i}$}

Our purpose is now to establish the equality 
\begin{equation}
\widetilde{f}_{i}^{\ast }(\psi )=\widehat{f}_{i}(\psi )\text{ for any }\psi
\in \Psi _{e}.  \label{equality}
\end{equation}
This is achieved by showing that the relations $\widetilde{f}_{i}^{\ast } 
\widetilde{f}_{j}\psi =\widetilde{f}_{j}\widetilde{f}_{i}^{\ast }\psi $ and $
\widehat{f}_{i}\widetilde{f}_{j}\psi =\widetilde{f}_{j}\widehat{f}_{i}\psi $
are both equivalent to a very simple condition on $\psi $.

\begin{lemma}
\label{lem_tech}Put $i\in \mathbb{Z}/e\mathbb{Z}$.

\begin{enumerate}
\item Consider $\psi ,\chi \in \Psi _{e}$ such that $\psi =\widehat{f}
_{i}\chi $ and put $j\in \mathbb{Z}/e\mathbb{Z}$. Then we have: 
\begin{equation*}
l_{0,j}(\chi )\neq l_{0,j}(\psi )\Longleftrightarrow i=j,\ \widehat{a}
_{i}(\chi )=0\text{ and }a_{i}(\chi )=1.
\end{equation*}

\item Consider $\psi ,\chi \in \Psi _{e}$ such that $\psi =\widetilde{f}
_{i}\chi $ and put $j\in \mathbb{Z}/e\mathbb{Z}$. Then we have: 
\begin{equation*}
\widehat{l}_{0,j}(\chi )\neq \widehat{l}_{0,j}(\psi )\Longleftrightarrow
i=j,\ a_{i}(\chi )=0\text{ and }\widehat{a}_{i}(\chi )=1.
\end{equation*}
\end{enumerate}
\end{lemma}

\begin{proof}
1: Assume first $\widehat{l}_{0,i}(\chi )=\widehat{l}_{0}>1$. Hence $
\widehat{a}_{i}(\chi )>0$ and $\psi =\chi -[i+1,\cdots,i+\widehat{l}_{0}-1]+[i,\cdots,i+ 
\widehat{l}_{0}-1]$.\ If $j\notin \{(i+\widehat{l}_{0}-1)(\mathrm{mod}e),(i+ 
\widehat{l}_{0})(\mathrm{mod}e)\},$ neither $[i+1,\cdots,i+\widehat{l}_{0}-1]$ or $
[i,\cdots,i+\widehat{l}_{0}-1]$ are segments $A$ or $R$ for $j$. We have $
w_{j}(\psi )=w_{j}(\chi )$ and then $l_{0,j}(\chi )=l_{0,j}(\psi ).$ Thus we
can restrict ourselves to the cases $j\in \{(i+\widehat{l}_{0}-1)(\mathrm{
mod }e),(i+\widehat{l}_{0})(\mathrm{mod}e)\}$. We write 
\begin{multline*}
\widehat{w}_{i}(\chi )=\cdots \lbrack i,\cdots,i+\widehat{l}_{0}-2]^{m_{[i,\cdots,i+ 
\widehat{l}_{0}-2]}}[i+1,\cdots,i+\widehat{l}_{0}-1]^{m_{[i+1,\cdots,i+\widehat{l}_{0}-1]}}
\\
\lbrack i,\cdots,i+\widehat{l}_{0}-1]^{m_{[i,\cdots,i+\widehat{l}_{0}-1]}}[i+1,\cdots,i+\widehat{
l }_{0}]^{m_{[i+1,\cdots,i+\widehat{l}_{0}]}}\cdots
\end{multline*}
where we have only pictured the segments of length $\widehat{l}_{0}-1$ and $
\widehat{l}_{0}$ of $\widehat{w}_{i}(\chi )$. Since $\psi =\widehat{f}
_{i}\chi $, we have 
\begin{multline*}
\widehat{w}_{i}(\psi )=\cdots \lbrack i,\cdots,i+\widehat{l}_{0}-2]^{m_{[i,\cdots,i+ 
\widehat{l}_{0}-2]}}[i+1,\cdots,i+\widehat{l}_{0}-1]^{m_{[i+1,\cdots,i+\widehat{l}
_{0}-1]}-1} \\
\lbrack i,\cdots,i+\widehat{l}_{0}-1]^{m_{[i,\cdots,i+\widehat{l}_{0}-1]}+1}[i+1,\cdots,i+ 
\widehat{l}_{0}]^{m_{[i+1,\cdots,i+\widehat{l}_{0}]}}\cdots
\end{multline*}
In particular, by (\ref{act_hat_RA}), we must have $m_{[i,\cdots,i+\widehat{l}
_{0}-2]}<m_{[i+1,\cdots,i+\widehat{l}_{0}-1]}$ and $m_{[i,\cdots,i+\widehat{l}_{0}-1]}\geq
m_{[i+1,\cdots,i+\widehat{l}_{0}]}$.

\noindent When $j=(i+\widehat{l}_{0}-1)(\mathrm{mod}e),$ $[i+1,\cdots,i+\widehat{l}
_{0}-1]$ and $[i,\cdots,i+\widehat{l}_{0}-1]$ are of type $R$ for $j$. Hence, by
considering only the segments of lengths $\widehat{l}_{0}-1$ and $\widehat{l}
_{0},$ we can write 
\begin{multline*}
w_{j}(\chi )=\cdots \lbrack i+1,\cdots,i+\widehat{l}_{0}-1]^{m_{[i+1,\cdots,i+\widehat{l}
_{0}-1]}}[i,\cdots,i+\widehat{l}_{0}-2]^{m_{[i,\cdots,i+\widehat{l}_{0}-2]}} \\
\lbrack i,\cdots,i+\widehat{l}_{0}-1]^{m_{[i,\cdots,i+\widehat{l}_{0}-1]}}[i-1,\cdots,i+\widehat{
l }_{0}-2]^{m_{[i-1,\cdots,i+\widehat{l}_{0}-2]}}\cdots
\end{multline*}
and 
\begin{multline*}
w_{j}(\psi )=\cdots \lbrack i+1,\cdots,i+\widehat{l}_{0}-1]^{m_{[i+1,\cdots,i+\widehat{l}
_{0}-1]}-1}[i,\cdots,i+\widehat{l}_{0}-2]^{m_{[i,\cdots,i+\widehat{l}_{0}-2]}} \\
\lbrack i,\cdots,i+\widehat{l}_{0}-1]^{m_{[i,\cdots,i+\widehat{l}_{0}-1]+1}}[i-1,\cdots,i+ 
\widehat{l}_{0}-2]^{m_{[i-1,\cdots,i+\widehat{l}_{0}-2]}}\cdots
\end{multline*}
Since $m_{[i,\cdots,i+\widehat{l}_{0}-2]}<m_{[i+1,\cdots,i+\widehat{l}_{0}-1]}$, the
cancellation procedures of the factors $RA$ in $w_{j}(\chi )$ and $
w_{j}(\psi )$ yield the same final word. Hence $\widetilde{w}_{j}(\psi )= 
\widetilde{w}_{j}(\chi )$ and we have also $l_{0,j}(\chi )=l_{0,j}(\psi )=1.$

\noindent When $j=(i+\widehat{l}_{0})(\mathrm{mod}e)$, $[i+1,\cdots,i+\widehat{l}
_{0}-1]$ and $[i,\cdots,i+\widehat{l}_{0}-1]$ are of type $A$ for $j$. We obtain
also $\widetilde{w}_{j}(\psi )=\widetilde{w}_{j}(\chi )$ by considering the
segments of lengths $\widehat{l}_{0}-1$ and $\widehat{l}_{0}$. Thus $
l_{0,j}(\chi )=l_{0,j}(\psi )$.

\noindent Observe that we have always $\widetilde{w}_{j}(\psi )=\widetilde{w}
_{j}(\chi )$ for any $j\in \mathbb{Z}/e\mathbb{Z}$ when $\widehat{l}_{0}>1.$
In particular 
\begin{equation}
\widehat{a}_{i}(\chi )>0\Longrightarrow a_{j}(\chi )=a_{j}(\widehat{f}
_{i}\chi )\text{ for any }j\in \mathbb{Z}/e\mathbb{Z}.  \label{case_l>1}
\end{equation}

Now assume $\widehat{l}_{0}=1,$ that is $\psi =\chi +[i]$. Write 
\begin{equation*}
\widehat{w}_{i}(\chi )=[i]^{m_{[i]}}[i+1]^{m_{[i+1]}}\cdots \text{ and } 
\widehat{w}_{i}(\psi )=[i]^{m_{[i]}+1}[i+1]^{m_{[i+1]}}\cdots
\end{equation*}
with $m_{[i]}\geq m_{[i+1]}$.

\noindent When $j=i+1(\mathrm{mod}e),$ $[i]$ is of type $A$ for $j$ and $
[i+1]$ is of type $R$. Thus we can write 
\begin{equation*}
w_{j}(\chi )=[i+1]^{m_{[i+1]}}[i]^{m_{[i]}}\cdots \text{ and }w_{j}(\psi
)=[i+1]^{m_{[i+1]}}[i]^{m_{[i]}+1}\cdots
\end{equation*}
Since $m_{[i]}\geq m_{[i+1]},$ the rightmost segments $A$ in $\widetilde{w}
_{j}(\chi )$ and $\widetilde{w}_{j}(\psi )$ are the same and we have yet $
l_{0,j}(\chi )=l_{0,j}(\psi ).$

\noindent When $j=i(\mathrm{mod}e),$ $[i]$ is of type $R$ for $j.$ Observe
that $\widehat{a}_{i}(\chi )=0$. Set $\widetilde{w}_{i}(\chi )=A^{a_{i}(\chi
)}R^{r_{i}(\chi )}.$ Then $\widetilde{w}_{i}(\psi )$ is obtained by applying
the cancellation procedure of the factors $RA$ to the word $w=RA^{a_{i}(\chi
)}R^{r_{i}(\chi )}$. Clearly, $l_{0,j}(\chi )\neq l_{0,j}(\psi )$ if and
only if $a_{i}(\chi )=1$ for in this case we have $l_{0,j}(\chi )>1$ and $
l_{0,j}(\psi )=1$. This proves assertion 1.

2: The arguments are similar to those used in the proof of 1.
\end{proof}

\begin{proposition}
\label{prop_com_fi_fihat}For any $\chi \in \Psi _{e}$ and $i,j\in \mathbb{Z}
/e\mathbb{Z}$ ,\ we have $\widetilde{f}_{i}\widehat{f}_{j}\chi \neq \widehat{
f}_{j}\widetilde{f}_{i}\chi \Longleftrightarrow i=j$ and $a_{i}(\chi )+ 
\widehat{a}_{i}(\chi )=1.$
\end{proposition}

\begin{proof}
Assume $i\neq j$ or, $i=j$ and $a_{i}(\chi )+\widehat{a}_{i}(\chi )>1.$ Then
by assertions 1 and 2 of the previous lemma, we have $\widehat{l}_{0,j}(\chi
)=\widehat{l}_{0,j}(\widetilde{f}_{i}\chi )=\widehat{l}_{0}$ and $
l_{0,i}(\chi )=l_{0,i}(\widehat{f}_{j}\chi )=l_{0}.$ Hence 
\begin{equation*}
\widetilde{f}_{i}\widehat{f}_{j}\chi =\chi +[j;\widehat{l}
_{0})+(l_{0};i]-[j+1;\widehat{l}_{0}-1)-(l_{0}-1;i-1]=\widehat{f}_{j} 
\widetilde{f}_{i}\chi
\end{equation*}
with $[j+1;\widehat{l}_{0}-1)=\emptyset $ if $\widehat{l}_{0}=1$ and $
(l_{0}-1;i-1]=\emptyset $ if $l_{0}=1.$

Now, assume $i=j$, $a_{i}(\chi )=1$ and $\widehat{a}_{i}(\chi )=0.$ We have 
\begin{equation*}
\widetilde{f}_{i}\widehat{f}_{i}\chi =\chi +2[i]\text{ and }\widehat{f}_{i} 
\widetilde{f}_{i}\chi =\chi +[i]+[i-l_{0}+1,\cdots,i]-[i-l_{0}+1,\cdots,i-1]
\end{equation*}
with $l_{0}=l_{0,i}(\chi )>1.$ Similarly, if we assume $i=j$, $a_{i}(\chi
)=1 $ and $\widehat{a}_{i}(\chi )=0,$ we obtain 
\begin{equation*}
\widehat{f}_{i}\widetilde{f}_{i}\chi =\chi +2[i]\text{ and }\widetilde{f}
_{i} \widehat{f}_{i}\chi =\chi +[i]+[i+1,\cdots,i+\widehat{l}_{0}-1]-[i,\cdots,i+\widehat{l
} _{0}-1]
\end{equation*}
with $\widehat{l}_{0}=\widehat{l}_{0,i}(\chi )>1.$ In both cases, $
\widetilde{f}_{i}\widehat{f}_{j}\chi \neq \widehat{f}_{i}\widetilde{f}
_{i}\chi $ which completes the proof. Observe that we then have 
\begin{equation}
\widetilde{f}_{i}\widehat{f}_{i}\chi =(\widehat{f}_{i})^{2}\chi \text{ and } 
\widehat{f}_{i}\widetilde{f}_{i}\chi =(\widetilde{f}_{i})^{2}\chi .
\label{case_ad^}
\end{equation}
\end{proof}

\begin{proposition}
\label{prop_com_fi_fi*}Consider $\psi \in \Psi _{e}$ and $i,j\in \mathbb{Z}
/e \mathbb{Z}$.

\begin{enumerate}
\item If $i\neq j,$ we have $\widetilde{f}_{i}\widetilde{f}_{j}^{\ast }\psi
= \widetilde{f}_{j}^{\ast }\widetilde{f}_{i}\psi $.

\item If $i=j,$ set $m=\varepsilon _{i}^{\ast }(\psi ).$ Then $\widetilde{f}
_{i}\widetilde{f}_{i}^{\ast }\psi \neq \widetilde{f}_{i}^{\ast }\widetilde{f}
_{i}\psi \Longleftrightarrow \varphi _{i}((\widetilde{e}_{i}^{\ast
})^{m}\psi )=\varepsilon _{i}^{\ast }(\psi )+1.$
\end{enumerate}
\end{proposition}

\begin{proof}
1: This is a classical property of crystals.\ Write $\theta _{j}(\psi )=( 
\widetilde{e}_{j}^{\ast })^{m}\psi \otimes \widetilde{f}_{j}^{m}b_{j}$ where 
$m=\varepsilon _{j}^{\ast }(\psi ).$ Then by (\ref{TENS1}), we have $\theta
_{j}(\widetilde{f}_{i}\psi )=\widetilde{f}_{i}(\widetilde{e}_{j}^{\ast
})^{m}\psi \otimes \widetilde{f}_{j}^{m}b_{j}$ for $i\neq j.$ By Proposition 
\ref{Prop_ac*}, we obtain $\theta _{j}(\widetilde{f}_{j}^{\ast }\widetilde{f}
_{i}\psi )=\widetilde{f}_{i}(\widetilde{e}_{j}^{\ast })^{m}\psi \otimes 
\widetilde{f}_{j}^{m+1}b_{j}.$ We have also $\theta _{j}(\widetilde{f}
_{j}^{\ast }\psi )=(\widetilde{e}_{j}^{\ast })^{m}\psi \otimes \widetilde{f}
_{j}^{m+1}b_{j}$ and since $i\neq j,$ this yields $\theta _{j}(\widetilde{f}
_{i}\widetilde{f}_{j}^{\ast }\psi )=\widetilde{f}_{i}(\widetilde{e}
_{j}^{\ast })^{m}\psi \otimes \widetilde{f}_{j}^{m+1}b_{j}$. Hence $\theta
_{j}(\widetilde{f}_{i}\widetilde{f}_{j}^{\ast }\psi )=\theta _{j}(\widetilde{
f}_{j}^{\ast }\widetilde{f}_{i}\psi )$ and we have $\widetilde{f}_{i} 
\widetilde{f}_{j}^{\ast }\psi =\widetilde{f}_{j}^{\ast }\widetilde{f}
_{i}\psi $ because $\theta _{j}$ is an embedding of crystals.

2: We derive by using the same arguments 
\begin{equation*}
\theta _{i}(\widetilde{f}_{i}\widetilde{f}_{i}^{\ast }\psi )=\left\{ 
\begin{array}{l}
\widetilde{f}_{i}(\widetilde{e}_{i}^{\ast })^{m}\psi \otimes \widetilde{f}
_{i}^{m+1}b_{i}\text{ if }\varphi _{i}((\widetilde{e}_{i}^{\ast })^{m}\psi
)>m+1, \\ 
(\widetilde{e}_{i}^{\ast })^{m}\psi \otimes \widetilde{f}_{i}^{m+2}b_{i} 
\text{ if }\varphi _{i}((\widetilde{e}_{i}^{\ast })^{m}\psi )\leq m+1.
\end{array}
\right.
\end{equation*}
We have also 
\begin{equation*}
\theta _{i}(\widetilde{f}_{i}^{\ast }\widetilde{f}_{i}\psi )=\left\{ 
\begin{array}{l}
\widetilde{f}_{i}(\widetilde{e}_{i}^{\ast })^{m}\psi \otimes \widetilde{f}
_{i}^{m+1}b_{i}\text{ if }\varphi _{i}((\widetilde{e}_{i}^{\ast })^{m}\psi
)>m, \\ 
(\widetilde{e}_{i}^{\ast })^{m}\psi \otimes \widetilde{f}_{i}^{m+2}b_{i} 
\text{ if }\varphi _{i}((\widetilde{e}_{i}^{\ast })^{m}\psi )\leq m.
\end{array}
\right.
\end{equation*}
Thus we obtain $\theta _{i}(\widetilde{f}_{i}\widetilde{f}_{i}^{\ast }\psi
)=\theta _{i}(\widetilde{f}_{i}^{\ast }\widetilde{f}_{i}\psi )$ except when $
\varphi _{i}((\widetilde{e}_{i}^{\ast })^{m}\psi )=m+1.$ Observe that we
have in this case 
\begin{equation}
\widetilde{f}_{i}\widetilde{f}_{i}^{\ast }\psi =(\widetilde{f}_{i}^{\ast
})^{2}\psi \neq (\widetilde{f}_{i})^{2}\psi =\widetilde{f}_{i}^{\ast } 
\widetilde{f}_{i}\psi .  \label{case_ad*}
\end{equation}
\end{proof}

\begin{lemma}
\label{lem_phi}Consider $\psi \in \Psi _{e}$ and $i\in \mathbb{Z}/e\mathbb{Z}
$. Set $\mathrm{wt}(\psi )=\sum_{i\in \mathbb{Z}/e\mathbb{Z}}\mathrm{wt}
_{i}(\psi )\Lambda _{i}$. Then we have

\begin{enumerate}
\item $\mathrm{wt}_{i}(\psi )=a_{i}(\psi )-r_{i}(\psi )+\widehat{a}_{i}(\psi
)-\widehat{r}_{i}(\psi ),$

\item $\varphi _{i}(\psi )=a_{i}(\psi )+\widehat{a}_{i}(\psi )-\widehat{r}
_{i}(\psi )$.
\end{enumerate}
\end{lemma}

\begin{proof}
1: Set 
\begin{equation*}
\psi =\sum_{l\geq 1}m_{(l;i]}(l;i]=\sum_{l\geq 1}m_{[i;l)}[i;l).
\end{equation*}
During the cancellation procedure described in \S\ \ref{subsec_cancel},
pairs of segments $(R,A)$ or $(\widehat{R},\widehat{A})$ are deleted. Thus
assertion 1 is equivalent to the equality $\mathrm{wt}_{i}(\psi )=\Delta
_{i}(\psi )$ where 
\begin{equation}
\Delta _{i}(\psi )=\sum_{l\geq 1}m_{(l;i-1]}-m_{(l;i]}+\sum_{l\geq
1}m_{[i+1;l)}-m_{[i;l)}.  \label{equa}
\end{equation}
We proceed by induction on $\left| \psi \right| .$ For $\psi =\emptyset ,$ (\ref{equa}) is satisfied.\ Now assume the equalities (\ref{equa}) hold for
any $i\in \mathbb{Z}/e\mathbb{Z}$ with $\left| \psi \right| =n$.\ Set $\psi
^{\prime }=\widetilde{f}_{j}\psi $. We have $\mathrm{wt}(\psi ^{\prime })= 
\mathrm{wt}(\psi )-\alpha _{j}$. Since $\alpha _{j}=2\Lambda _{j}-\Lambda
_{j+1}-\Lambda _{j-1}$, this gives 
\begin{equation}
\mathrm{wt}_{i}(\psi ^{\prime })=\left\{ 
\begin{array}{ll}
\mathrm{wt}_{i}(\psi ) & \text{ if }i\notin \{j-1,j,j+1\}, \\ 
\mathrm{wt}_{i}(\psi )-2 & \text{ if }i=j, \\ 
\mathrm{wt}_{i}(\psi )+1 & \text{ if }i\in \{j-1,j+1\}.
\end{array}
\right.  \label{rel}
\end{equation}
The multisegment $\psi ^{\prime }$ is obtained by adding the segments $[j]$
to $\psi $ or by replacing a segment $(l-1;j-1]$ in $\psi $ by the segment $
(l,j].$ This shows that the relations (\ref{rel}) are also satisfied by the $
\Delta _{i}(\psi ^{\prime })$'s. Hence $\Delta _{i}(\psi ^{\prime })=\mathrm{
\ wt}_{i}(\psi ^{\prime })$ for any $i\in \mathbb{Z}/e\mathbb{Z}$.

2: By (\ref{weight}), we have $\mathrm{wt}_{i}(\psi )=\varphi _{i}(\psi
)-\varepsilon _{i}(\psi )$. Lemma \ref{lem_util} then gives $\mathrm{wt}
_{i}(\psi )=\varphi _{i}(\psi )-r_{i}(\psi )$. Thus $\varphi _{i}(\psi
)=a_{i}(\psi )+\widehat{a}_{i}(\psi )-\widehat{r}_{i}(\psi )$ by 1.
\end{proof}

\bigskip

To prove (\ref{equality}), we are going to proceed by induction on $n=\left|
\psi \right| .$ We easily check that $\widetilde{f}_{i}^{\ast }(\boldsymbol{
\ \emptyset})=\widehat{f}_{i}(\boldsymbol{\emptyset})=[i]$ for the empty
multisegment $\boldsymbol{\emptyset}.\;$Now assume that $\widetilde{f}
_{i}^{\ast }(\psi )=\widehat{f}_{i}(\psi )$ holds for any multisegment $\psi
\in \Psi _{e}$ such that $\left| \psi \right| \leq n.\;$

\begin{proposition}
\label{prop_fund}Under the previous induction hypothesis we have for any $
\chi \in \Psi _{e}$ such that $\left| \chi \right| \leq n$ 
\begin{equation}
\widetilde{f}_{i}\widehat{f}_{j}\chi \neq \widehat{f}_{j}\widetilde{f}
_{i}\chi \Longleftrightarrow \widetilde{f}_{i}\widetilde{f}_{j}^{\ast }\chi
\neq \widetilde{f}_{j}^{\ast }\widetilde{f}_{i}\chi \Longleftrightarrow i=j 
\text{ and }a_{i}(\chi )+\widehat{a}_{i}(\chi )=1.  \label{equ_rec}
\end{equation}
\end{proposition}

\begin{proof}
Note first that the proposition does not directly follows from the induction
hypothesis for $\left| \widetilde{f}_{j}\chi \right| =n+1.$ By this
induction hypothesis, we have $(\widetilde{e}_{i}^{\ast })^{\varepsilon
_{i}^{\ast }(\chi )}\chi =(\widehat{e}_{i})^{\widehat{r}_{i}(\chi )}\chi .$
Set $\chi ^{\prime }=(\widehat{e}_{i})^{\widehat{r}_{j}(\chi )}\chi .$
Assertion 2 of Lemma \ref{lem_phi} gives 
\begin{equation*}
\varphi _{i}(\chi ^{\prime })=a_{i}(\chi ^{\prime })+\widehat{a}_{i}(\chi
^{\prime })-\widehat{r}_{i}(\chi ^{\prime })=a_{i}(\chi ^{\prime })+\widehat{
a}_{i}(\chi )+\widehat{r}_{i}(\chi )
\end{equation*}
for we have $\widehat{r}_{i}(\chi ^{\prime })=0$ and $\widehat{a}_{i}(\chi
^{\prime })=\widehat{a}_{i}(\chi )+\widehat{r}_{i}(\chi )$. Observe that $
\varepsilon _{i}^{\ast }(\chi )=\widehat{r}_{i}(\chi )$ be the induction
hypothesis. Moreover, we have $a_{i}(\chi )=a_{i}(\chi ^{\prime })$ by (\ref
{case_l>1}) since $\widehat{a}_{i}(\varphi )>0$ for any $\varphi =(\widehat{
e }_{i})^{a}\chi $ with $a\in \{1,\ldots ,\widehat{r}_{j}(\chi )\}$.\ This
gives the equivalences 
\begin{equation*}
\varphi _{i}(\chi ^{\prime })=\varepsilon _{i}^{\ast }(\chi
)+1\Longleftrightarrow a_{i}(\chi ^{\prime })+\widehat{a}_{i}(\chi
)=1\Longleftrightarrow a_{i}(\chi )+\widehat{a}_{i}(\chi )=1.
\end{equation*}
Now Propositions \ref{prop_com_fi_fihat} and \ref{prop_com_fi_fi*} yields ( 
\ref{equ_rec}).
\end{proof}

We immediately derive by using (\ref{rel_tilde_hat}) : {\ We are now able to
prove the main result of this section.}

\begin{theorem}
For any multisegment $\psi \in \Psi _{e}$ and any $j\in \mathbb{Z}/e\mathbb{
Z },$ we have $\widetilde{f}_{j}^{\ast }(\psi )=\widehat{f}_{j}(\psi ).$
\end{theorem}

\begin{proof}
{\ We argue by induction on $n=\left| \psi \right| $. We already know that
for all $j\in \mathbb{Z}/e\mathbb{Z}$, we have $\widetilde{f}_{j}^{\ast
}(\emptyset )=\widehat{f}_{j}(\emptyset )=[j].$ Now assume $\widetilde{f}
_{j}^{\ast }\chi =\widehat{f}_{j}\chi $ for any $j\in \mathbb{Z}/e\mathbb{Z}$
and any $\chi \in \Psi _{e}$ such that $\left| \chi \right| \leq n.$
Consider }$\psi \in \Psi _{e}$ such that $\left| \psi \right| =n+1.\ ${There
exists $i\in \mathbb{Z}/e\mathbb{Z}$ and $\chi \in \Psi _{e}$ such that $
\psi =\widetilde{f}_{i}\chi $ and }$\left| {\chi }\right| =n.$

When $i\neq j$ or $a_{i}(\chi )+\widehat{a}_{i}(\chi )>1,$ we have by
Proposition \ref{prop_fund} $\widetilde{f}_{j}^{\ast }\psi =\widetilde{f}
_{j}^{\ast }\widetilde{f}_{i}\chi =\widetilde{f}_{i}\widetilde{f}_{j}^{\ast
}\chi .$ By our induction hypothesis, we can thus write $\widetilde{f}
_{j}^{\ast }\psi =\widetilde{f}_{i}\widehat{f}_{j}\chi .$ Since $a_{i}(\chi
)+\widehat{a}_{i}(\chi )>1,$ this finally gives $\widetilde{f}_{j}^{\ast
}\psi =\widehat{f}_{j}\widetilde{f}_{i}\chi =\widehat{f}_{j}\psi .$

\noindent When $i=j$ and $a_{i}(\chi )+\widehat{a}_{i}(\chi )=1,$ we obtain $
\widetilde{f}_{i}^{\ast }\psi =\widetilde{f}_{i}^{\ast }\widetilde{f}
_{i}\chi =\widetilde{f}_{i}{}^{2}\psi $ by (\ref{case_ad*}). Similarly, we
have $\widehat{f}_{i}\psi =\widehat{f}_{i}\widetilde{f}_{i}\chi =\widetilde{
f }_{i}{}^{2}\psi $ by (\ref{case_ad^}). Thus $\widetilde{f}_{i}^{\ast }\psi
= \widehat{f}_{i}\psi $ which completes the proof.
\end{proof}

\begin{remark}
Theorem \ref{Th_f*=fhat} and Proposition \ref{prop_fund} notably imply the
equivalence 
\begin{equation*}
\widetilde{f}_{i}\widetilde{f}_{j}^{\ast }\psi =\widetilde{f}_{j}^{\ast } 
\widetilde{f}_{i}\psi \Leftrightarrow i\neq j\text{ or }a_{i}(\psi )+ 
\widehat{a}_{i}(\psi )>1
\end{equation*}
for any $\psi \in \Psi _{e}$.
\end{remark}

\section{Affine Hecke algebra of type $A$ and Ariki-Koike algebras}

\subsection{Identification of simple modules}

Let $\mathbf{v}=(v_{0},\ldots ,v_{l-1})\in \mathcal{V}_{l}$. The Ariki-Koike
algebra $\mathcal{H}_{n}^{\mathbf{v}}(\xi )$ is the quotient $\mathcal{H}
_{n}^{a}(\xi )/I_{\mathbf{v}}$ where $I_{\mathbf{v}}=\langle P_{\mathbf{v}
}=\prod_{i=0}^{l-1}(X_{1}-\xi ^{v_{i}})\rangle $. Then each simple $\mathcal{
\ H}_{n}^{\mathbf{v}}(\xi )$-module is isomorphic to a simple $\mathcal{H}
_{n}^{a}(\xi )$-module of $\mathrm{Mod}_{n}^{a}$. By the Specht module
theory developed by Dipper, James and Mathas \cite{DJM}, the simple $\mathcal{H}_{n}^{\mathbf{v}}(\xi )$-modules are 
parametrized by certain $l$-partitions of $n$ called Kleshchev multipartitions. Let $\Phi _{e}^{K}( 
\mathbf{v})$ be the set of Kleshchev $l$-partitions. Given $\mathbf{\mu }$
in $\Phi _{e}^{K}(\mathbf{v})$, write $D^{\mathbf{\mu }}$ for the simple $
\mathcal{H}_{n}^{\mathbf{v}}(\xi )$-module associated to $\mathbf{\mu }$
under this parametrization. In fact, we shall need in the sequel the
parametrization of the simple $\mathcal{H}_{n}^{\mathbf{v}}(\xi )$-modules
by FLOTW $l$-partitions. The correspondence between the parametrizations by
Kleshchev and FLOTW $l$-partitions has been detailed in \cite{JL}. In
particular, the bijection $\Gamma :\Phi _{e}(\mathbf{v})\rightarrow \Phi
_{e}^{K}(\mathbf{v})$ is an isomorphism of $\mathcal{U}_{v}$-crystals which
can easily be made explicit. This means that, given any ${\boldsymbol{
\lambda }}$ in $\Phi _{e}(\mathbf{v}),$ we can compute $\Gamma ({\boldsymbol{
\lambda} )}$ directly from ${\boldsymbol{\lambda}}$ without using the
crystal structures on $\Phi _{e}(\mathbf{v})$ and $\Phi _{e}^{K}(\mathbf{v})$
. We then set $\widetilde{D}^{{\boldsymbol{\lambda}}}=D^{\Gamma ({
\boldsymbol{\lambda})}}$. This gives the natural labelling 
\begin{equation*}
\text{Irr}(\mathcal{H}_{n}^{\mathbf{v}}(\xi ))=\{\widetilde{D}^{{\boldsymbol{
\lambda}}}\ |\ {\boldsymbol{\lambda}}\in \Phi _{e}(\mathbf{v})\}
\end{equation*}
{which coincides with the parametrization of the simple $\mathcal{H}_{n}^{ 
\mathbf{v}}(\xi )$-modules in terms of Geck-Rouquier canonical basic set
obtained in \cite{J}.}

\noindent The simple $\mathcal{H}_{n}^{a}(\xi)$-module $L_{\psi}$ with $
\psi\in\Psi_{e}$ isomorphic to $\widetilde{D}^{{\boldsymbol{\lambda}}}$ is
given by the following theorem (see \cite[Thm 6.2]{AJL}).

\begin{theorem}
Let ${\boldsymbol{\lambda}}\in \Phi _{e}(\mathbf{v})$ then 
\begin{equation*}
\widetilde{D}^{{\boldsymbol{\lambda}}}\simeq L_{f_{\mathbf{v}}({\boldsymbol{
\lambda}})}
\end{equation*}
where $f_{v}$ is the crystal embedding of Theorem \ref{Th_fv}.
\end{theorem}

Conversely, given any simple $\mathcal{H}_{n}^{a}(\xi )$-module $L_{\psi }$,
it is natural to search for the Ariki-Koike algebras $\mathcal{H}_{n}^{ 
\mathbf{v}}(\xi )$ with $\mathbf{v}$ in $\mathcal{V}_{l}$ and the simple $
\mathcal{H}_{n}^{\mathbf{v}}(\xi )$-module $\widetilde{D}^{{\boldsymbol{\lambda}}}$ such that $\widetilde{D}^{{\boldsymbol{\lambda}}}\simeq L_{\psi
} $. This problem turns out to be more complicated. Indeed we have first to
determine all the multicharges $\mathbf{v}$ such that $f_{\mathbf{v}
}^{-1}(\psi )\neq \emptyset $ and next we need to compute the $l$-partition $
{\boldsymbol{\lambda}}$ satisfying $f_{\mathbf{v}}({\boldsymbol{\lambda}}
)=\psi $. Note that ${\boldsymbol{\lambda }}$ is necessarily unique for a
given $\mathbf{v}$ since $f_{\mathbf{v}}$ is injective. {We will then say
that $\mathbf{v}$ is an \emph{admissible multicharge} with respect to $\psi $
when }$f_{\mathbf{v}}^{-1}(\psi )\neq \emptyset .$ Then {${\boldsymbol{\
\lambda=}}$}$f_{\mathbf{v}}^{-1}(\psi )${\ is its corresponding \emph{\
admissible multipartition}.} In the next paragraphs, we shall completely
solve the problem of determining all the admissible multicharges and FLOTW
multipartitions associated to an aperiodic multisegment $\psi $. To obtain
the corresponding Kleshchev multipartition, it then suffices to apply $
\Gamma $.

\subsection{Admissible multicharges}

\label{algo} Let $\psi \in \Psi _{e}$. To find a multicharge $\mathbf{v}$
such that $f_{\mathbf{v}}^{-1}(\psi )\neq \emptyset $, we compute $
\varepsilon _{i}^{\ast }(\psi )$ for all $i\in \mathbb{Z}/e\mathbb{Z}$ by
using the equality $\varepsilon _{i}^{\ast }(\psi )=\widehat{r}_{i}(\psi )$
established in Theorem \ref{Th_f*=fhat}. For a multicharge $\mathbf{v}$ in $
\mathcal{V}_{l}$ and $i\in \mathbb{Z}/e\mathbb{Z}$, let $\kappa _{i}(\mathbf{
\ v})$ be the nonnegative integers such that 
\begin{equation*}
\mathbf{v}=(\underbrace{0,...,0}_{\kappa _{0}(\mathbf{v})},\underbrace{
1,...,1}_{\kappa _{1}(\mathbf{v})},...,\underbrace{e-1,...,e-1}_{\kappa
_{e-1}(\mathbf{v})}).
\end{equation*}
Then we have 
\begin{equation*}
f_{\mathbf{v}}^{-1}(\psi )\neq \emptyset \iff \forall i\in \mathbb{Z}/e 
\mathbb{Z},\ \kappa _{i}(\mathbf{v})\geq \varepsilon _{i}^{\ast }(\psi ).
\end{equation*}
Observe that the multicharge $\mathbf{v}(\psi )$ with $\kappa _{i}(\mathbf{v}
(\psi ))=\varepsilon _{i}^{\ast }(\psi )$ (defined at the end of \S \ref
{subsec_FLOW}) is the multicharge of minimal level among all the admissible
multicharges. It is of particular interest for the computation of the
involution $\sharp $ as we shall see in \S \ref{algo}.

\subsection{Admissible multipartitions}

\label{mpartition} Consider $\psi \in \Psi _{e}$, $l\in \mathbb{N}$ and an
admissible multicharge $\mathbf{v}\in \mathcal{V}_{l}$ with respect to $\psi 
$. The aim of this section is to give a simple procedure for computing the
admissible $l$-partition $\mathbf{{\boldsymbol{\lambda}}}\in \Phi _{e}( 
\mathbf{v})$ associated to $\psi $ (i.e. such that $f_{\mathbf{v}}({{
\boldsymbol{\lambda}}})=\psi $).

\noindent We begin with a general lemma on FLOTW $l$-partitions. Consider $
\mathbf{v\in}\mathcal{V}_{l}$ and ${\boldsymbol{\lambda}}\in{\Phi}_{e}( 
\mathbf{v})$ a non-empty $l$-partition. Let $m$ be the length of the minimal
non zero part of ${\boldsymbol{\lambda}}$. Let ${\boldsymbol{\mu}}$ be the $
l $-partition obtained by deleting in ${\boldsymbol{\lambda}}$ the parts of
length $m$.

\begin{lemma}
\label{lem_reduc}The $l$-partition ${\boldsymbol{\mu}}$ belongs to ${\Phi }
_{e}(\mathbf{v})$.
\end{lemma}

\begin{proof}
Assume that ${\boldsymbol{\mu}}\notin \Phi _{e}(\mathbf{v})$. Then one of
the following situations happens.

\begin{itemize}
\item[(i)] There exists $c\in \{0,1,\ldots ,l-1\}$ and $i\in \mathbb{N}$
such that $\mu _{i}^{c}<\mu _{i+v_{c+1}-v_{c}}^{c+1}$. This implies in
particular that $\mu _{i+v_{c+1}-v_{c}}^{c+1}\neq 0$. Since ${\boldsymbol{{\ 
\boldsymbol{\lambda}}}}$ belongs to $\Phi _{e}(\mathbf{v})$, we have $
\lambda _{i}^{c}\geq \lambda _{i+v_{c+1}-v_{c}}^{c+1}$. Thus $\lambda
_{i+v_{c+1}-v_{c}}^{c+1}=\mu _{i+v_{c+1}-v_{c}}^{c+1}$, $\lambda _{i}^{c}=m$
and $\mu _{i}^{c}=0$. We have $\lambda _{i+v_{c+1}-v_{c}}^{c+1}\neq 0$ and $
\lambda _{i+v_{c+1}-v_{c}}^{c+1}\leq m$. This contradicts the fact that ${\ 
\boldsymbol{{\boldsymbol{\mu}}}}$ is obtained from ${\boldsymbol{{\ 
\boldsymbol{\lambda}}}}$ by deleting the minimal nonzero parts.

\item[(ii)] There exists $i\in \mathbb{N}$ such that $\mu _{i}^{l-1}<\mu
_{i+v_{0}-v_{l-1}+e}^{0}$ . We obtain a contradiction similarly.
\end{itemize}
\end{proof}

\bigskip

For $\psi \in \Psi _{e}$, define $l_{1}>\cdots >l_{r}>0$ as the decreasing
sequence of (distinct) lengths of the segments appearing in $\psi .$ For any 
$t=1,\ldots ,r,$ write $a_{t}$ for the number of segments in $\psi $ with
length $l_{t}.$ Set $\psi _{0}=\boldsymbol{\emptyset }$ the empty
multisegment and $\psi _{r}=\psi .$ For any $1\leq t\leq r-1$ let $\psi _{t}$
be the multisegment obtained from $\psi $ by deleting successively the
segments of length $l_{t+1},\ldots ,l_{r}.$ Clearly $\psi _{t}$ is aperiodic.

Assume $\mathbf{{\boldsymbol{\lambda }}}\in \Phi _{e}(\mathbf{v})$ is
associated to $\psi $.\ Since $f_{\mathbf{v}}({{\boldsymbol{\lambda }}}
)=\psi $, the sequence $l_{1}>\cdots >l_{r}$ is also the decreasing sequence
of distinct parts appearing in $\mathbf{{\boldsymbol{\lambda }}}.$ Moreover,
for ant $t=1,\ldots ,r,$ $\mathbf{{\boldsymbol{\lambda }}}$ contains $a_{t}$
parts equal to $l_{t}.$ Set ${\boldsymbol{\lambda }}[r]={\boldsymbol{\lambda 
}}$.\ Let ${\boldsymbol{\lambda }}[t],$ $t=0,\ldots ,r-1$ be the $l$
-partitions obtained by deleting successively the parts of lengths $
l_{r},\ldots ,l_{t+1}$ in ${\boldsymbol{\lambda }}$.\ By Lemma \ref
{lem_reduc}, the $l$-partitions ${\boldsymbol{\lambda }}[t],$ $t=0,\ldots
,r-1$ all belong to $\Phi _{e}(\mathbf{v})$. Since $f_{\mathbf{v}}({\ 
\boldsymbol{\lambda }})=\psi $ we must also have 
\begin{equation*}
f_{\mathbf{v}}({\boldsymbol{\lambda }}[t])=\psi _{t}\text{ for any }
t=0,\ldots ,r
\end{equation*}
by definition of the map $f_{\mathbf{v}}$ (see (\ref{f_v})).

\bigskip

We are going to compute ${\boldsymbol{\lambda }}$ from $\psi $ by induction
on the lengths $l_{1}>\cdots >l_{r}>0$ of the segments of $\psi .$ To do
this, we have to determine the sequence of $l$-partitions ${\boldsymbol{\
\lambda }}[t],$ $t=0,\ldots ,r$ associated to the segments $\psi
_{t},t=0,\ldots ,r.$ We have ${\boldsymbol{\lambda }}[0]=\boldsymbol{\
\emptyset }.$ Thus, it suffices to explain how ${\boldsymbol{\lambda }}
[t+1]\in \Phi _{e}(\mathbf{v})$ can be obtained from ${\boldsymbol{\lambda }}
[t]\in \Phi _{e}(\mathbf{v})$.

\noindent The $l$-partition ${\boldsymbol{\lambda }}[t+1]$ is constructed by
adding $a_{t+1}$ parts of lengths $l_{t+1}$ to ${\boldsymbol{\lambda }}[t]$
such that the parts added give segments $[k_{t+1};l_{t+1})$ in the
correspondence (\ref{f_v}). Since the nonzero parts ${\boldsymbol{\lambda }}
[t]$ are greater to $l_{t+1}$, these new parts can only appear on the bottom
of the partitions composing ${\boldsymbol{\lambda }}[t]$. The procedure for
computing ${\boldsymbol{\lambda }}[t+1]$ from ${\boldsymbol{\lambda }}[t]$
can be decomposed in the following three steps.

\begin{enumerate}
\item For $c=0,\ldots ,l-1,$ consider the integers 
\begin{equation*}
i_{c}=\text{min}\{a\in \mathbb{N}\ |\ {\boldsymbol{\lambda }}[t]_{a}^{c}=0\},
\end{equation*}
that is the sequence of depths of the partitions appearing in ${\boldsymbol{
\ \lambda }}$.

\item Let $c_{1},c_{2},\ldots ,c_{p}\in \{0,\ldots ,l-1\}$ be such that 
\begin{equation*}
k_{t+1}\equiv 1-i_{c_{1}}+v_{c_{1}}\equiv \cdots \equiv
1-i_{c_{p}}+v_{c_{p}}(\text{mod }e),
\end{equation*}
with $p\geq a_{t+1}$. These integers must exist by (\ref{f_v}) for $f_{ 
\mathbf{v(\psi )}}^{-1}(\psi _{t+1})\neq \emptyset $. Without loss of
generality, we can assume $c_{1}\blacktriangleleft \cdots \blacktriangleleft
c_{p}$ where $\blacktriangleleft $ is the total order on $\{0,\ldots ,l-1\}$
such that 
\begin{equation}
c\blacktriangleleft c^{\prime }\Longleftrightarrow \left\{ 
\begin{array}{l}
\mathrm{(i)}:v_{c}-i_{c}<v_{c^{\prime }}-i_{c^{\prime }}\text{ or} \\ 
\mathrm{(ii)}:v_{c}-i_{c}=v_{c^{\prime }}-i_{c^{\prime }}\text{ and }
c<c^{\prime }\text{ as integers.}
\end{array}
\right.  \label{condition}
\end{equation}

\item The problem reduces to determine $a_{t+1}$ partitions among the
partitions $\lambda \lbrack t]^{c_{f}},$ $f=1,\ldots ,p$ which, once
completed with a part $l_{t+1}$, yield an $l$-partition of $\Phi _{e}( 
\mathbf{v})$. Set $S[t+1]=\{(c_{_{1}},i_{c_{1}}),\ldots
,(c_{a_{t+1}},i_{c_{a_{t+1}}})\}$.\ Let $\widehat{{\boldsymbol{\lambda }}}
[t+1]$ be the $l$-partition defined by 
\begin{equation*}
\widehat{{\boldsymbol{\lambda }}}[t+1]_{i}^{c}=\left\{ 
\begin{array}{ll}
{\boldsymbol{\lambda }}[t]_{i}^{c} & \text{ if }(c,i)\notin S[t+1], \\ 
l_{t+1} & \text{ if }(c,i)\in S[t+1].
\end{array}
\right.
\end{equation*}
So $\widehat{{\boldsymbol{\lambda }}}[t+1]$ is obtained by adding $a_{t+1}$
parts $l_{t+1}$ on the bottom of the partitions ${\boldsymbol{\lambda }}
[t]^{c}$ with $c\in \{c_{1},\ldots ,c_{a_{t+1}}\}.$ This means that the new
parts are added on the bottom of the $a_{t+1}$ first partitions considered
following (\ref{condition}).
\end{enumerate}

\begin{lemma}
With the above notation, $\widehat{{\boldsymbol{\lambda}}}[t+1]={\boldsymbol{
\ \lambda}[t+1]}$.
\end{lemma}

\begin{proof}
It suffices to prove that $\widehat{{\boldsymbol{\lambda }}}[t+1]$ belongs
to $\Phi _{e}(\mathbf{v})$. Indeed, this will give $f_{\mathbf{v}}(\widehat{{
\boldsymbol{\lambda }}}[t+1])=f_{\mathbf{v}}({\boldsymbol{\lambda }}
[t+1])=\psi _{t+1}$ and thus, $\widehat{{\boldsymbol{\lambda }}}[t+1]={\ 
\boldsymbol{\lambda }[t+1]}$ since $f_{\mathbf{v}}$ is an embedding$.\ $The
second condition to be a FLOTW $l$-partition is clearly satisfied for the
multisegment $\psi _{t+1}$ is aperiodic. We have to check that the first
condition also holds.

\noindent Assume that $\widehat{{\boldsymbol{\lambda}}}[t+1]$ does not
satisfy condition 1 of Definition \ref{f}.

\noindent Suppose first we have $\widehat{{\boldsymbol{\lambda}}}
[t+1]_{i}^{s}<\widehat{{\boldsymbol{\lambda}}}[t+1]_{i+v_{s+1}-v_{s}}^{s+1}$
where $s\in \{1,\ldots ,l-1\}$ and $i$ is a nonnegative integer. Since ${\ 
\boldsymbol{\lambda}}[t]\in \Phi _{e}(\mathbf{v})$, we have $\widehat{{\ 
\boldsymbol{\lambda}}}[t+1]_{i}^{s}={\boldsymbol{\lambda}}[t]_{i}^{s}=0$ and 
${\boldsymbol{\lambda}}[t]_{i+v_{s+1}-v_{s}}^{s+1}=0$, $\widehat{{\ 
\boldsymbol{\lambda}}}[t+1]_{i+v_{s+1}-v_{s}}^{s+1}=l_{t+1}$. Thus $
(s+1,i+v_{s+1}-v_{s})\in S[t+1].$ We have two cases to consider.

\begin{itemize}
\item Assume $i+v_{s+1}-v_{s}>1$ and $\widehat{{\boldsymbol{\lambda}}}
[t+1]_{i+v_{s+1}-v_{s}-1}^{s+1}={\boldsymbol{\lambda}}
[t]_{i+v_{s+1}-v_{s}-1}^{s+1}>0.$ Then $i=1$ or $\widehat{{\boldsymbol{\
\lambda}}}[t+1]_{i-1}^{s}\neq 0.$ Indeed we must have $\widehat{{\boldsymbol{
\ \lambda}}}[t+1]_{i-1}^{s}=\widehat{{\boldsymbol{\lambda}}}
[t]_{i-1}^{s}\geq {\ \boldsymbol{\lambda}}[t]_{i+v_{s+1}-v_{s}-1}^{s+1}$
because ${\boldsymbol{\ \lambda}}[t]$ belongs to $\Phi _{e}(\mathbf{v})$. We
have $(s+1,i+v_{s+1}-v_{s})\in S[t+1]$.$\;$In particular 
\begin{equation*}
k_{t+1}\equiv v_{s+1}-(i+v_{s+1}-v_{s})+1\equiv v_{s}-i+1(\mathrm{mod}\text{ 
}e).
\end{equation*}
Since ${\boldsymbol{\lambda}}[t]_{i}^{s}=0$, this means that $(s,i)\in
S[t+1].$ But this is a contradiction. Indeed by condition $\mathrm{(ii)}$ of
(\ref{condition}), we should have $\widehat{{\boldsymbol{\lambda}}}
[t+1]_{i}^{s}=l_{t+1}\neq 0$.

\item Assume $i+v_{s+1}-v_{s}=1$ then $i=1$ and we have $v_{s+1}=v_{s}$.
Thus $(s,i)\in S[t+1]$ and we derive a contradiction similarly.
\end{itemize}

\noindent Now suppose we have $\widehat{{\boldsymbol{\lambda}}}
[t+1]_{i}^{l-1}>\widehat{{\boldsymbol{\lambda}}}
[t+1]_{i+v_{0}-v_{l-1}+e}^{0} $. The proof is analogue.\ We obtain that $
(l-1,i)\in S[t+1]$ and $(0,i+v_{0}-v_{l-1}+e)\in S[t+1].\;$This contradicts
condition $\mathrm{(i)}$ of (\ref{condition}).
\end{proof}

\bigskip

By using the above procedure, we are now able to compute the $l$-partitions $
{\boldsymbol{\lambda}}[t],$ $t=1,\ldots,r$ from {$\psi$} and from its
associated admissible multicharge $\mathbf{v}$. This thus gives a recursive
algorithm for computing the admissible $l$-partition ${\boldsymbol{\lambda}}$
from {$\psi$}.

\subsection{Example}

Let $e=4$. We consider the following aperiodic multisegment 
\begin{equation*}
\psi =[0;6)+[0;5)+[3;5)+[1;4)+2[3;3)+[0;3)+[2;2)+[2;1).
\end{equation*}
We have $\widehat{w}_{0}(\psi )=\widehat{R}\widehat{A}\widehat{R}\widehat{R},
$ $\widehat{w}_{1}(\psi )=\widehat{A}\widehat{A}\widehat{R},$ $\widehat{w}
_{2}(\psi )= \widehat{R}\widehat{R}\widehat{A}\widehat{A}\widehat{A}$ and $
\widehat{w}_{3}(\psi )=\widehat{R}\widehat{R}\widehat{A}\widehat{R} \widehat{
R}\widehat{A}$. This gives 
\begin{equation*}
\varepsilon _{0}^{\ast }(\psi )=2,\ \varepsilon _{1}^{\ast }(\psi
)=1,\varepsilon _{2}^{\ast }(\psi )=0,\ \varepsilon _{3}^{\ast }(\psi )=0.
\end{equation*}
Thus the multicharge $(0,0,1)$ is an admissible multicharge. Actually this
is the one with minimal level. We now use the above algorithm to compute the
associated admissible $l$-partition ${\boldsymbol{\lambda}}$. Using the same
notation as above, we successively obtain 
\begin{equation*}
{\boldsymbol{\lambda}}[0]=(\emptyset ,\emptyset ),\qquad {\boldsymbol{
\lambda }}[1]=(6,\emptyset,\emptyset),\qquad {\boldsymbol{\lambda}}
[2]=(6.5,5,\emptyset), \qquad {\boldsymbol{\lambda}}[3]=(6.5,5,4),
\end{equation*}
\begin{equation*}
{\boldsymbol{\lambda}} [4]=(6.5,5.3,4.3.3),\quad {\boldsymbol{\lambda}}
[5]=(6.5.2,5.3,4.3.3),\quad {\boldsymbol{\lambda}}[6]=(6.5.2,5.3.1,4.3.3)
\end{equation*}
${\boldsymbol{\lambda}}=(6.5.2,5.3.1,4.3.3)$ is the admissible $3$-partition
associated to the multicharge $(0,0,1)$ and $\psi $. We easily check that 
\begin{equation*}
f_{(0,0,1)}(6.5.2,5.3.1,4.3.3)=\psi .
\end{equation*}
This means that the modules $L_{\psi }$ and $\widetilde{D}^{{\boldsymbol{\
\lambda}}}$ are isomorphic.

\noindent The multicharge $(0,0,1,2,3)$ is another example of an admissible
multicharge (with level $5$) and its associated admissible multipartition is 
${\boldsymbol{\lambda}}=(6.3,5.3,4.3,2,5.1)$ .

\section{Computation of the involution $\sharp\label{Sec_sharpinv}$}

\subsection{The generalized Mullineux involution}

\label{subsecMulli}

The two fold symmetry $i\longleftrightarrow-i$ defines a skew crystal
isomorphism from $B_{e}(\mathbf{v})$ to $B_{e}(\mathbf{v}^{\sharp})$ where $
\mathbf{v}=(v_{0},\ldots,v_{l-1})$ and $\mathbf{v}^{\sharp}=(-v_{l-1},
\ldots,-v_{0})$ belong to $\mathcal{V}_{l}$ (see (\ref{vFLOTW})). Given ${
\boldsymbol{\lambda}\in\Phi}_{e}(\mathbf{v}),$ write $m_{l}^{\mathbf{v}}({\boldsymbol{\lambda})}\in{\Phi}_{e}(\mathbf{v}^{\sharp})$ for the image of ${
\boldsymbol{\lambda}}$ under this skew isomorphism. In \cite{KF}, Ford and
Kleshchev proved that for $l=1$, the map $m_{l}^{\mathbf{v}}$ reduces to the
Mullineux involution $m_{1}$ on $e$-restricted partitions. Thus we call $
m_{l}^{\mathbf{v}}$ the generalized Mullineux involution.

\noindent By \S\ 5.1, the set $\Phi_{e}^{K}(\mathbf{v})$ of Kleshchev $l$
-partitions has also the structure of an affine crystal isomorphic to $
B_{e}( \mathbf{v})$. In particular the two fold symmetry $
i\longleftrightarrow -i$ also defines a bijection $m_{l}^{\mathbf{v},K}$
from $\Phi_{e}^{K}(\mathbf{v} )$ to $\Phi_{e}^{K}(\mathbf{v}^{\sharp})$. In 
\cite{JL2}, we gave an explicit procedure yielding $m_{l}^{\mathbf{v},K}$.
Given ${{\boldsymbol{\ \lambda}=(\lambda}}^{0},\ldots,\lambda^{l-1})\in
\Phi_{e}^{K}(\mathbf{v})$, the $l$-partition ${{\boldsymbol{\mu}=}}m_{l}^{
\mathbf{v},K}({{\boldsymbol{\lambda})}}$ is obtained by computing first 
\begin{equation*}
{{\boldsymbol{\nu}=}}(m_{1}(\lambda^{0}),\ldots,m_{1}(\lambda^{l-1}))
\end{equation*}
i.e. the $l$-partition obtained by applying the Mullineux map to each
partition of ${{\boldsymbol{\lambda}}}$. The $l$-partition ${{\boldsymbol{\
\nu }}}$ does not belong to $\Phi_{e}^{K}(\mathbf{v}^{\sharp})$ in general
and we have then to apply a straightening algorithm (detailed in \cite{JL2} 
\S 4.3) to obtain ${{\boldsymbol{\mu}}}$.

\noindent As already noted in \S 5.1, we have a bijection (in fact a crystal
isomorphism) 
\begin{equation*}
\Gamma:\Phi_{e}(\mathbf{v})\rightarrow\Phi_{e}^{K}(\mathbf{v})
\end{equation*}
which can be made explicit by using the results of \cite{JL}. This permits
to compute the map $m_{l}^{\mathbf{v}}$ since 
\begin{equation}
m_{l}^{\mathbf{v}}=\Gamma^{-1}\circ m_{l}^{\mathbf{v},K}\circ\Gamma.
\label{geneMull}
\end{equation}

\begin{remark}
\ 

\begin{enumerate}
\item The previous procedure yielding the generalized Mullineux map $m_{l}^{ 
\mathbf{v}}$ can be optimized. In particular the conjugation by the map $
\Gamma $ can be avoided. Nevertheless, the pattern of the computation
remains essentially the same : it uses the original Mullineux map $m_{1}$
and the results of \cite{JL} on affine crystal isomorphisms.\ Since it
requires some technical combinatorial developments which are not essential
for our purposes, we have chosen to omit it here.

\item Note also that in the case $e=\infty $, the map $\Gamma $ is the
identity and $m_{1}$ is simply the conjugation operation on the partitions.
As observed in \cite[\S 4.4]{JL2}, the algorithm for computing $m_{l}^{ 
\mathbf{v}}=m_{l}^{\mathbf{v},K}$ then considerably simplifies.
\end{enumerate}
\end{remark}

\subsection{The algorithm}

Let $\psi \in \Psi _{e}$ then, to compute $\psi ^{\sharp }$, we first
determine an admissible multicharge $\mathbf{v}$ with respect to $\psi $ and
the associated admissible multipartition ${{\boldsymbol{\lambda }}}$. Then
we apply the above algorithm to compute $m_{l}^{\mathbf{v}}({{\boldsymbol{\lambda }}}).$ It turns out that the complexity of this algorithm
considerably increases with the level of $\mathbf{v}$. Hence, the use of the
admissible multicharge $\mathbf{v}(\psi )$ with minimal level is preferable.
Let us summarize the different steps of the procedure we have to apply to
compute $\psi ^{\sharp }$ :

\begin{enumerate}
\item For $i\in \mathbb{Z}/e\mathbb{Z}$, we compute $\varepsilon _{i}^{\ast
}(\psi )$. To do this, we use Theorem \ref{Th_f*=fhat} which gives the
equalities $\varepsilon _{i}^{\ast }(\psi )=\widehat{r_{i}}(\psi )$ for all $
i\in \mathbb{Z}/e\mathbb{Z}$. We then put 
\begin{equation*}
\mathbf{v}(\psi )=(\underbrace{0,...,0}_{\varepsilon _{0}^{\ast }(\mathbf{v}
)},\underbrace{1,...,1}_{\varepsilon _{1}^{\ast }(\mathbf{v})},..., 
\underbrace{e-1,...,e-1}_{\varepsilon _{e-1}^{\ast }(\mathbf{v})}).
\end{equation*}
By Theorem \ref{Th_f*=fhat}, $\mathbf{v}(\psi )$ is an admissible
multicharge in $\mathcal{V}_{l}$.

\item Using \S \ref{mpartition}, we compute the admissible FLOTW
multipartition \underline{$\boldsymbol{\lambda}$}$(\psi )$ with respect to $
\mathbf{v}({\boldsymbol{\lambda}})$ and ${\psi }$.

\item Using \S \ref{subsecMulli}, we compute the image $m_{l}^{\mathbf{v}}($ 
\underline{$\boldsymbol{\lambda}$}$(\psi ))$ of \underline{$\boldsymbol{\lambda}$}$(\psi )$ under the generalized Mullineux involution,

\item We finally obtain the aperiodic multisegment $\psi ^{\sharp }=f_{ 
\mathbf{v}(\psi )}(m_{l}^{\mathbf{v}}(${\underline{$\boldsymbol{\lambda }$}$
(\psi $}${)).}$ using \S \ref{subsec_FLOW}.
\end{enumerate}

\begin{remark}
In the case where $e=\infty ,$ our algorithm for computing the Zelevinsky
involution is essentially equivalent to that described by Moeglin and
Waldspuger in \cite{MW} except we use multipartitions rather than
multisegments.
\end{remark}

\section{Further remarks}

\subsection{Computation of the Kashiwara involution}

We have established in Section \ref{Sec_hat=star}, that the crystal
operators $\widetilde{f}_{i}^{\ast }$ and $\widehat{f}_{i}$ coincide for any 
$i\in \mathbb{Z}/e\mathbb{Z}$. Given $\psi \in \Psi _{e}$, we can thus
compute $\psi ^{\ast }$ by determining a path $\psi =\widetilde{f}
_{i_{1}}\cdots \widetilde{f}_{i_{n}}\boldsymbol{\emptyset}$ in $B_{e}(\infty
)$ from the empty multisegment to $\psi .$ We have then $\psi ^{\ast }= 
\widehat{f}_{i_{1}}\cdots \widehat{f}_{i_{n}}\boldsymbol{\emptyset}.$

\noindent By combining the algorithm described in Section \ref{Sec_sharpinv}
for computing the involution $\sharp $ with the relation $\ast =\tau =\rho
\circ \sharp $ , we obtain another procedure computing $\ast $ on $
B_{e}(\infty )$. This procedure does not required the determination of a
path in the crystal $B_{e}(\infty ).$

\begin{example}
Assume $e=2,$ the involution $\sharp $ is nothing but the identity and, thus
the Kashiwara involution coincides with $\rho $ on $B_{e}(\infty ).$
\end{example}

\subsection{Crystal commutor for $B_{e}(\mathbf{v})\otimes B_{e}(\mathbf{v}
^{\prime})$}

In \cite{KT}, Kamnitzer and Tingley introduced a crystal commutor for any
symmetrizable Kac-Moody algebra. Recall that a crystal commutor for $B_{e}( 
\mathbf{v})\otimes B_{e}(\mathbf{v}^{\prime })$ is a crystal isomorphism 
\begin{equation*}
\sigma _{\mathbf{v},\mathbf{v}^{\prime }}:B_{e}(\mathbf{v})\otimes B_{e}( 
\mathbf{v}^{\prime })\rightarrow B_{e}(\mathbf{v}^{\prime })\otimes B_{e}( 
\mathbf{v}).
\end{equation*}
This isomorphism is unique if and only if $B_{e}(\mathbf{v})\otimes B_{e}( 
\mathbf{v}^{\prime })$ does not contain two isomorphic connected components
that is, if the decomposition of the corresponding tensor product is without
multiplicity. Such a crystal commutor is defined by specifying the images of
the highest weight vertices of $B_{e}(\mathbf{v})\otimes B_{e}(\mathbf{v}
^{\prime }).$ It is easy to verify by using (\ref{TENS2}) that the highest
weight vertices of $B_{e}(\mathbf{v})\otimes B_{e}(\mathbf{v}^{\prime })$
are precisely the vertices of the form $\boldsymbol{\emptyset}\otimes 
\boldsymbol{\lambda}$ with $\boldsymbol{\lambda}\in B_{e}(\mathbf{v}^{\prime
})$ such that $\varepsilon _{i}(\boldsymbol{\lambda})\leq r_{i}$ for any $
i\in \mathbb{Z}/e\mathbb{Z}$ ($r_{i}$ is the number of coordinates in $
\mathbf{v}$ equal to $i$). Denote by $\mathcal{H}_{\mathbf{v},\mathbf{v}
^{\prime }}$ the set of highest weight vertices in $B_{e}(\mathbf{v})\otimes
B_{e}(\mathbf{v}^{\prime })$.

\noindent For any $\boldsymbol{\lambda}\in B_{e}(\mathbf{v}),$ write for
short $\boldsymbol{\lambda}^{\ast }=$\underline{$\boldsymbol{\lambda}$}$(f_{ 
\mathbf{v}^{\prime }}(\boldsymbol{\lambda})^{\ast })$ (see the definition of 
\underline{$\boldsymbol{\lambda}$} below (\ref{v(m)})). Since $\ast $ is an
involution, we have $\boldsymbol{\lambda}^{\ast }\in B_{e}(\mathbf{w})$
where $\mathbf{w}\in \mathcal{V}_{l}$ is the multicharge having $\varepsilon
_{i}(\boldsymbol{\lambda})$ coordinates equal to $i$ and level $l=\sum_{i\in 
\mathbb{Z}/e\mathbb{Z}}\varepsilon _{i}(\boldsymbol{\lambda })$. The
condition $\varepsilon _{i}(\boldsymbol{\lambda})\leq r_{i}$ for any $i\in 
\mathbb{Z}/e\mathbb{Z}$ then implies that $\boldsymbol{\lambda}^{\ast }\in
B_{e}(\mathbf{v})$. We have the following theorem which is the main result
of \cite{KT}.

\begin{theorem}
\ 

\begin{enumerate}
\item Assume $\boldsymbol{\emptyset}\otimes \boldsymbol{\lambda}\in \mathcal{
\ H}_{\mathbf{v},\mathbf{v}^{\prime }}$. Then $\boldsymbol{\emptyset}\otimes 
\boldsymbol{\lambda }^{\ast }\in \mathcal{H}_{\mathbf{v}^{\prime },\mathbf{v}
}.$

\item The map 
\begin{equation*}
\sigma _{\mathbf{v},\mathbf{v}^{\prime }}:\left\{ 
\begin{array}{l}
\mathcal{H}_{\mathbf{v},\mathbf{v}^{\prime }}\rightarrow \mathcal{H}_{ 
\mathbf{v}^{\prime },\mathbf{v}} \\ 
\boldsymbol{\emptyset}\otimes \boldsymbol{\lambda}\longmapsto \boldsymbol{\
\emptyset}\otimes \boldsymbol{\lambda}^{\ast }
\end{array}
\right.
\end{equation*}
defines a crystal commutor for $B_{e}(\mathbf{v})\otimes B_{e}(\mathbf{v}
^{\prime })$.
\end{enumerate}
\end{theorem}

\noindent The results established in Sections \ref{Sec_hat=star} and \ref
{Sec_sharpinv} then permit to compute the crystal commutor of Kamnitzer and
Tingley for affine type $A$ crystals.

\begin{example}
Assume $e=2.$ Then the crystal commutor $\sigma _{\mathbf{v},\mathbf{v}
^{\prime }}$ satisfies $\sigma _{\mathbf{v},\mathbf{v}^{\prime }}( 
\boldsymbol{\emptyset }\otimes \boldsymbol{\lambda }\mathbf{)=(}\boldsymbol{
\ \emptyset }\otimes $\underline{$\boldsymbol{\lambda }$}$(\rho \circ f_{ 
\mathbf{v}^{\prime }}(\boldsymbol{\lambda }))$ for any $\boldsymbol{\
\emptyset }\otimes \boldsymbol{\lambda }\mathbf{\in }\mathcal{H}_{\mathbf{v}
, \mathbf{v}^{\prime }}$.
\end{example}

\bigskip

\noindent \textbf{Acknowledgments.} The authors are very grateful to the
anonymous referee for his/her careful reading and for having pointing
several inaccuracies in a previous version of this paper.


\begin{thebibliography}{99}
\bibitem{Ari3} \textsc{S.~Ariki}, Representations of quantum groups and
combinatorics of Young tableaux, A. M. S. University Lecture Series \textbf{
\ 26} (2002).

\bibitem{Ari4} \textsc{S.~Ariki}, Proof of the modular branching rule for
Ariki-Koike algebras, Journal of Algebra 306 (2006), 290-300.

\bibitem{AJL} \textsc{S.~Ariki, N.~ Jacon and C.~Lecouvey}, The modular
branching rule for affine Hecke algebras of type $A$, preprint, \textsf{\
http://arxiv.org/abs/0808.3915}.

\bibitem{BO} \textsc{C.~Bessenrodt and J.B.~Olsson}, On residue symbols and
the Mullineux conjecture. J. Algebraic Combin. \textbf{7} (1998), no. 3,
227--251.

\bibitem{BK} \textsc{J.~Brundan and A.~Kleshchev}, Hecke-Clifford
superalgebras, crystals of type $A_{2\ell }^{(2)}$ and modular branching
rule for $\widehat{S}_{n}$, Representation Theory \textbf{5} (2001), 317-403.

\bibitem{CG} \textsc{N.~Chriss and V.~Ginzburg}, Representation Theory and
Complex Geometry, Birkh\"{a}user, 1997.

\bibitem{DJM} \textsc{R.~Dipper, G.~James, A.~Mathas, }Cyclotomic $q$-Schur
Algebras, Mathematische\ Zeitschrift\ 229, (1998), 385-416.

\bibitem{FLOTW} \textsc{O.~Foda, B.~Leclerc, M.~Okado, J.-Y.~ Thibon and
T.~Welsh}, Branching functions of $A_{n-1}^{(1)}$ and Jantzen-Seitz problem
for Ariki-Koike algebras, Advances in Math. \textbf{141} (1999), 322--365.

\bibitem{KF} \textsc{B.~Ford and A.~Kleshchev}, A proof of the Mullineux
conjecture. Math. Z. \textbf{226} (1997), no. 2, 267--308.

\bibitem{Groj} \textsc{I.~Grojnowski}, Affine $\widehat{sl}_{p}$ controls
the modular representation theory of the symmetric group and related Hecke
algebras, \textsf{http://arXiv:math/9907129.}

\bibitem{Guo} \textsc{Jin Yun Guo}, The Hall polynomials of a cyclic serial
algebra, Comm. Algebra \textbf{23} (1995), 743--751.

\bibitem{IM} \textsc{N.~Iwahori and H.~Matsumoto}, On some Bruhat
decomposition of the Hecke rings of $p$-adic Chevalley groups, Publ. Math.
I.H.E.S, \textbf{25} (1965), 237--280.

\bibitem{J} \textsc{N. Jacon,} On the parametrization of the simple modules
for Ariki-Koike algebras at roots of unity. J. Math. Kyoto Univ. \textbf{44}
(2004), no. 4, 729--767.

\bibitem{JL} \textsc{N. Jacon and C. Lecouvey, }Crystal isomorphisms for
irreducible highest weight $\mathcal{U}_{v}(\widehat{\mathfrak{sl}}_{e})$
-modules of higher level, to appear in Algebras and Rep. Theory, \textsf{\
http://arxiv.org/abs/0706.0680}.

\bibitem{JL2} \textsc{N. Jacon and C.Lecouvey, }On the Mullineux involution
for Ariki-Koike algebras, to appear in Journal of Algebra (sec. Comp.
Algebra), \textsf{http://arxiv.org/abs/0804.0478}.

\bibitem{jim} \textsc{M. Jimbo, K.\ C.\ Misra, T.\ Miwa and M.\ Okado, }
Combinatorics of representations of $\mathcal{U}_{q}(\widehat{sl}(n))$ at $
q=0,$ Communication in Mathematical Physics\ \textbf{136} (1991), 543-566.

\bibitem{KT} \textsc{J.~Kamnitzer, P.~Tingley}, A definition of the crystal
commutor using Kashiwara's involution, to appear in the Journal of Algebraic
Combinatorics.

\bibitem{Kas} \textsc{M. Kashiwara,}\textit{\ }On crystal bases,\textit{\ }
Canadian Mathematical Society, Conference Proceedings, \textbf{16} (1995),
155-197.

\bibitem{KZ} \textsc{H. Knight and A. Zelevinsky,} Representations of
quivers of type $A$ and the multisegment duality. Adv. Math. 117 (1996), no.
2, 273--293.

\bibitem{LTV} \textsc{B.~Leclerc, J-Y.~Thibon, E.~Vasserot}, Zelevinsky's
involution at roots of unity. J. Reine Angew. Math. 513 (1999), 33--51.

\bibitem{Lu} \textsc{G.~Lusztig}, Canonical basis and Hall algebras,
Lectures at the summer school on representation theory and algebraic
geometry, Montreal (1997).

\bibitem{MW} \textsc{C.~Moeglin and J-L.~Waldspurger}, Sur l'involution de
Zelevinsky, J. Reine Angew. Math. \textbf{372} (1986) 136-177.

\bibitem{M} \textsc{G.~Mullineux}, Bijections of $p$-regular partitions and $
p$-modular irreducibles of the symmetric groups. J. London Math. Soc. (2) 
\textbf{20} (1979), no. 1, 60--66.

\bibitem{Rich} \textsc{M.~Richards,} Some decomposition numbers for Hecke
algebras of general linear groups. Math. Proc. Cambridge Philos. Soc. 
\textbf{119} (1996), no. 3, 383--402.

\bibitem{R} \textsc{C.~Ringel}, The composition algebra of a cyclic quiver,
Proc. London Math. Soc. \textbf{66} (1993) 507-537.

\bibitem{U} \textsc{D. Uglov}, Canonical bases of higher-level $q$-deformed
Fock spaces and Kazhdan-Lusztig polynomials. Physical combinatorics (Kyoto,
1999), 249--299, Progr. Math., \textbf{191}, Birkh\"{a}user Boston, Boston,
MA, 2000.

\bibitem{V} \textsc{M.~Vazirani} Parametrizing Hecke algebra modules:
Bernstein-Zelevinsky multisegments, Kleshchev multipartitions, and crystal
graphs, Transform. Groups 7 (2002), 267-303.

\bibitem{Vi} \textsc{M-F. Vigneras}, A propos d'une conjecture de Langlands
modulaire, in Finite reductive groups: related structures and
representations, Cabanes, Marc (ed.), Birkhauser 1997.

\bibitem{Ze} \textsc{A. Zelevinsky} Induced representations of reductive ${\
p}$-adic groups. II. On irreducible representations of ${\ GL}(n)$. Ann.
Sci. \'{E}cole Norm. Sup. (4) 13 (1980), no. 2, 165--210.
\end{thebibliography}
\end{document}